\documentclass[a4paper,12pt,centertags]{amsart}
\usepackage{mathrsfs}
\usepackage{amssymb}
\usepackage{bbm}
\usepackage[a4paper,centering]{geometry}
\usepackage[pdfdisplaydoctitle,colorlinks,breaklinks,urlcolor=blue,linkcolor=blue,citecolor=blue]{hyperref}
\usepackage{mathtools}
\usepackage{mathscinet}
\usepackage{xspace}
\usepackage{tikz}
\newtheorem{theorem}{Theorem}[section]
\newtheorem{lemma}[theorem]{Lemma}
\newtheorem{proposition}[theorem]{Proposition}
\newtheorem{corollary}[theorem]{Corollary}
\theoremstyle{definition}

\newtheorem{assumption}[theorem]{Assumption}
\theoremstyle{remark}
\newtheorem{remark}[theorem]{Remark}
\newtheorem{example}[theorem]{Example}
\numberwithin{equation}{section}
\DeclareMathOperator{\Var}{Var}
\DeclareMathOperator{\Supp}{Supp}
\newcommand{\e}{\operatorname{e}}
\newcommand{\im}{\mathrm{i}}
\newcommand{\E}{\mathbb{E}}
\newcommand{\R}{\mathbb{R}}
\newcommand{\Z}{\mathbb{Z}}

\newcommand{\Tb}{\mathbb{T}}
\newcommand{\Dc}{\mathcal{D}}
\newcommand{\Fc}{\mathcal{F}}
\newcommand{\Mcc}{\mathcal{M}}

\newcommand{\Cs}{\mathscr{C}}
\newcommand{\Fs}{\mathscr{F}}
\newcommand{\Xbo}{\mathbf{X}}
\newcommand{\Prob}{\mathbb{P}}
\newcommand{\uno}{\mathbbm{1}}
\newcommand{\eqdef}{\vcentcolon=}
\newcommand{\qedef}{=\vcentcolon}
\newcommand{\scalar}[1]{\langle #1 \rangle}
\newcommand{\apperr}{\texttt{\bfseries\color{red} Ae}\xspace}
\newcommand{\prober}{\texttt{\bfseries\color{red} Pe}\xspace}
\newcommand{\Ae}{\emph{approximation error}\xspace}
\newcommand{\Pe}{\emph{probabilistic estimate}\xspace}
\DeclareMathOperator{\plt}{\text{\textcircled{\tiny $<$}}}
\DeclareMathOperator{\pgt}{\text{\textcircled{\tiny $>$}}}
\DeclareMathOperator{\peq}{\text{\textcircled{\tiny $=$}}}
\newcommand{\memo}[1]{ 
  \ensuremath{
    \framebox{\tiny\text{\kern-2pt\textsf{#1}}\kern-2pt}
  }
  \xspace
}
\newcommand{\arxiv}[1]{\href{http://arxiv.org/abs/#1}{\scriptsize arXiv: #1}}
\def\MRnum#1\empty{#1}
\renewcommand{\MRhref}[2]{%
  \href{http://www.ams.org/mathscinet-getitem?mr=#1}{#2}
}
\renewcommand{\MR}[1]{
  \relax\ifhmode\unskip\space\fi
  \MRhref{\MRnum#1\empty}{\texttt{\Tiny[MR\MRnum#1\empty]}}
}
\tikzset{
  every path/.style={thick},
  not/.style={circle,fill=black,draw=black,inner sep=0pt,minimum size=0.6mm},
  wot/.style={circle,fill=white,draw=black,inner sep=0pt,minimum size=0.9mm},
}
\protected\def\ga{
  \tikz[scale=0.3]{
    \draw[white] (-0.3,0)  -- (0.3,0);
    \draw (0,1) node[not] {} -- (0,0);
  }
}
\protected\def\gb{
  \tikz[scale=0.3]{
    \draw (-0.4,1) node[not] {} -- (0,0);
    \draw  (0.4,1) node[not] {} -- (0,0);
  }
}
\protected\def\gc{
  \tikz[scale=0.3]{
    \draw (-0.5,0.9) node[not] {} -- (0,0);
    \draw      (0,1) node[not] {} -- (0,0);
    \draw  (0.5,0.9) node[not] {} -- (0,0);
  }
}
\protected\def\goc{
  \tikz[baseline=1,scale=0.3]{
    \draw (0,0.5) node {} -- (0,0);
    \draw (-0.5,1.1) node[not] {} -- (0,0.5);
    \draw    (0,1.2) node[not] {} -- (0,0.5);
    \draw  (0.5,1.1) node[not] {} -- (0,0.5);
  }
}
\hypersetup{%
  pdftitle={A simple method for the existence of a density for stochastic evolutions with rough coefficients},
  pdfauthor={M. Romito},
  pdfcreator={M. Romito}
}
\begin{document}
  \title[A simple method for the existence of a density]
    {A simple method for the existence of a density
    for stochastic evolutions with rough coefficients}
  \author[M. Romito]{Marco Romito}
  \address{Dipartimento di Matematica, Universit\`a di Pisa, Largo Bruno Pontecorvo 5, I--56127 Pisa, Italia}
  \email{\href{mailto:marco.romito@unipi.it}{marco.romito@unipi.it}}
  \urladdr{\url{http://people.dm.unipi.it/romito}}
  \date{October 10, 2017}
  \begin{abstract}
    We extend the validity of a simple method for the existence of a
    density for stochastic differential equations, first introduced
    in \cite{DebRom2014}, by proving local estimate for the
    density, existence for the density with summable drift, and
    by improving the regularity of the density.
  \end{abstract}
  \maketitle
\section{Introduction}

The purpose of this paper is to illustrate and improve a simple
but effective method to prove existence and minimal regularity
of a density for the solutions of stochastic differential
equations with non--regular coefficients. The strong points
of the method are its simplicity, its flexibility and
the dimension--free nature of the regularity estimates.
Indeed, the simple method has been first introduced in
\cite{DebRom2014} (see also
\cite{Rom2014a,Rom2016,Rom2016a,Rom2016b}), widely
extending an idea in \cite{FouPri2010}, to prove
existence of densities for finite dimensional
approximations of an infinite dimensional
stochastic equation. The method has been later
used, see for instance
\cite{DebFou2013,Fou2015,SanSus2015,SanSus2015p,Alt2017}.

Indeed, the method we are discussing and extending
fits into the general problem of studying existence
and regularity of densities for stochastic equations
with non--smooth coefficients. The problem has
raised some interest recently. In addition to
the aforementioned \cite{FouPri2010,DebRom2014},
we would like to mention the approach in
\cite{BalCar2012,BalCar2014,BalCar2016}
based on some analytic criteria in
spaces of Orlicz type, and interpolation.
In \cite{HayKohYuk2013,HayKohYuk2014}
the author get rid of the drift with
a Girsanov transformation, and then use the
Malliavin calculus. Malliavin calculus is
also used in \cite{Kus2010}.
An atypical method
based on optimization is instead
introduced in \cite{BanKru2016,BanKru2017}.
Different approaches, based on an
explicit representation of the
density (and so giving typically
results as lower and upper bounds on
the density, rather than regularity)
have been given in
\cite{Hua2015,KohLi2016,KnoKul2017}.
Finally, \cite{MeyPro2010} proves Malliavin
differentiability of solutions of stochastic
equations with non--differentiable
coefficients. For a PDE approach see
for instance \cite{LebLio2008}.

In this paper, as well as in most of the aforementioned papers
(some with non--Gaussian noise), we focus on the following
``toy problem'',
\[
  dX_t
    = b(t,X_t)\,dt + \sigma(t,X_t)\,dB_t
\]
and in Section~\ref{s:bulk} we illustrate our simple method
for the existence of a density for the solutions of
the equation above, under the assumptions
$b\in L^\infty(\R^d)$, $\sigma\in C^\beta(\R^d;\R^{d\times d'})$,
for some $\beta\in(0,1)$, $d,d'\geq1$, and
$\sigma(y)\sigma(y)^\star\geq\delta I$,
with $\delta>0$ (weaker assumptions will
be discussed in the subsequent sections).
The method uses the idea of an
auxiliary (much) simpler process introduced
in \cite{FouPri2010}, together with
a smoothing lemma (see Appendix~\ref{s:smoothing})
for duality in Besov spaces. 

In this paper we use the simple method illustrated
in Section~\ref{s:bulk} to extend the scope and
the conclusions of the method itself. It is
of foremost importance to notice that, even
though we prove our results on the toy
problem above, the results themselves are not
difficult to extend to other contexts,
such as equations driven by non--Gaussian noise,
path--dependent equations, stochastic PDEs.
We shall illustrate some of these
examples in Section~\ref{s:apps}.

We first show in Section~\ref{s:morereg1}
that a higher order of approximation
provides, whenever it is compatible with
the regularity of the coefficients, a higher
regularity for the density.

Even though the results if Section~\ref{s:bulk}
assume uniform ellipticity of the diffusion,
it is not difficult to adapt the method to
both the cases of a singular diffusion matrix
as in Section~\ref{s:nonelliptic}, and of
a hypo-elliptic diffusion. In the latter case
we only discuss a simple example to illustrate
the ideas. Unfortunately, but not unexpectedly,
stronger regularity assumptions are necessary
here, strong enough that in principle the problem
may be amenable by a more standard approach such as
the Malliavin calculus.

In Section~\ref{s:local} we prove a local version
of the results of the simple method introduced
in Section~\ref{s:bulk}, to take into account
the case when the coefficients are not globally
regular or globally bounded.

In Section~\ref{s:rougher} we discuss the case
of rougher coefficients (with respect to
the assumptions of Section~\ref{s:bulk},
namely bounded drift and H\"older covariance).
In Section~\ref{s:rougherdrift} we prove
existence of a density for $L^p$ drifts,
with $p$ larger than the dimension of the state
spaces. Unfortunately we have not been able
to lower the regularity requirements for
the diffusion. We briefly discuss the issues
in Section~\ref{s:rougherdiffusion}.

Finally, in Section~\ref{s:morereg2} we
slightly improve on the summability
index of the Besov spaces for the
regularity of the density. This is
not yet a satisfactory result, since
in the basic case of Section~\ref{s:bulk}
one would expect H\"older regularity.
We can achieve H\"older regularity
only in dimension one (see
Remark~\ref{r:holder}).

Our results aim to be explanatory, so for instance
we will not mix two different lines of
development of the method (for instance, we
will not consider local results, as in
Section~\ref{s:local}, with rough
drift, as in Section~\ref{s:rougher},
or with singular diffusion, etc.).

In addition with the results explained so far
for the simple toy model, in Section~\ref{s:apps}
we discuss a series of examples,
\begin{itemize}
  \item path--dependent stochastic equations,
  \item improvements over \cite{DebFou2013}
    for a class of stochastic
    equations driven by $\alpha$--stable noise,
  \item a singular stochastic PDE.
\end{itemize} 
These examples
should convince of the power and flexibility
of our simple method. We notice in particular
that the dimension--free nature of the regularity
obtained is well suited for infinite
dimensional problem (as well as in general
for singular problems). It is sufficient
indeed to apply the simple method to
a series of approximating problems,
to obtain uniform estimate in a Besov
space with small but positive regularity.
This ensures uniform integrability and thus
convergence of the densities to the
density of the limit problem.

Finally, in Appendix~\ref{s:smoothing}
we introduce our functional analytic
framework, with the definition of all
the function spaces we use throughout
the paper, and we prove the crucial
smoothing results for laws of
random variables.
\section{The core idea}\label{s:bulk}

In this section we introduce the core idea around the
simple method which is the main theme of the paper.
The idea appears implicitly in \cite{DebRom2014}
(and later, explicitly, in \cite{DebFou2013}).
We repeat it here so that it will be the starting
point of our improvements. We will moreover
make the dependence on the time when the density
is evaluated and on the initial condition
explicit.

To get the gist of the idea, we focus on the simple
toy model,
\begin{equation}\label{e:toy}
  dX_t
    = b(X_t)\,dt + \sigma(X_t)\,dB_t,
\end{equation}
where $b\in L^\infty(\R^d)$,
$\sigma\in C_b^\beta(\R^d;\R^{d\times d'})$,
with $\beta\in(0,1)$, and $(B_t)_{t\geq0}$
is a $d'$-dimensional Brownian motion.
Assume moreover that
\begin{equation}\label{e:bulknonsingular}
  \text{there is $\delta>0$ such that}\qquad
  \mathop{det}(\sigma(y)\sigma^\star(y))\geq\delta>0
  \qquad\text{for all }y\in\R^d. 
\end{equation}
In other words the diffusion
coefficient is non--degenerate.

Our basic tool here is the smoothing
Lemma~\ref{l:smoothing1}. Fix an integer
$m\geq1$ large and a function $\phi$ in
$\Cs^\alpha_b(\R^d)$,
with $\alpha\in(0,1)$ to be chosen later.
Our aim is to estimate $\E[\Delta_h^m\phi(X_t)]$
and capture the regularizing effect of the
density. To this end, consider a number
$0<\epsilon<t\wedge 1$ and the auxiliary process
\begin{equation}\label{e:auxiliary}
  Y_s^\epsilon =
    \begin{cases}
      X_s,
        &\qquad s\leq t-\epsilon,\\
      X_{t-\epsilon} + \sigma(X_{t-\epsilon})(B_s-B_{t-\epsilon}),
        &\qquad s\geq t-\epsilon.
    \end{cases}
\end{equation}
We decompose the quantity $\E[\Delta_h^m\phi(X_t)]$
in two terms, the \Ae,
\begin{equation}\label{apperr}
  \apperr
    \eqdef \E[\Delta_h^m\phi(X_t)]
      - \E[\Delta_h^m\phi(Y^\epsilon_t)],
\end{equation}
and the \Pe,
\begin{equation}\label{prober}
  \prober
    \eqdef \E[\Delta_h^m\phi(Y^\epsilon_t)].
\end{equation}
We will address (and denote) the two fundamental quantities
we have defined as \apperr and \prober in the rest of the paper.

For the first term we use the regularity of the test
function $\phi$,
\[
  \apperr
    = \E[\Delta_h^m\phi(X_t)-\Delta_h^m\phi(Y_t^\epsilon)]
    \lesssim [\phi]_{\Cs^\alpha_b}\E[|X_t-Y_t^\epsilon|^\alpha]
    \lesssim [\phi]_{\Cs^\alpha_b}
      \E[|X_t-Y_t^\epsilon|^2]^{\frac\alpha2},
\]
where
\[
  X_t - Y_t^\epsilon
    = \int_{t-\epsilon}^t b(X_r)\,dr
      + \int_{t-\epsilon}^t (\sigma(X_r) - \sigma(X_{t-\epsilon})\,dB_r
\]
and, by standard estimates on stochastic equations,
we have,
\begin{equation}\label{e:bulkae}
  \begin{aligned}
    \E[|X_t-Y_t^\epsilon|^2]
      &\lesssim \|b\|_{L^\infty}^2\epsilon^2
        + [\sigma]^2_{C^\beta}\int_{t-\epsilon}^t
        \E[|X_r-X_{t-\epsilon}|^{2\beta}]\,dr\\
      &\lesssim \|b\|_{L^\infty}^2\epsilon^2 + [\sigma]^2_{C^\beta}
        (\|b\|_{L^\infty}^2+\|\sigma\|_{L^\infty}^2)^\beta\bigr)
        \epsilon^{1+\beta},
  \end{aligned}
\end{equation}
or, in other words, $\apperr\lesssim\epsilon^{\frac12\alpha(1+\beta)}$.
For the second term we condition over the history
$\Fs_{t-\epsilon}$ up to time $t-\epsilon$, as in
\cite{FouPri2010},
\[
  \prober
    = \E[\E[\Delta_h^m\phi(Y^\epsilon_t)|\Fs_{t-\epsilon}]]
    = \E[\E[\Delta_h^m\phi
      (y+\sigma(y)\tilde B_\epsilon)]_{y=X_{t-\epsilon}}]
\]
since, given $\Fs_{t-\epsilon}$, $Y^\epsilon_t$ is
a Brownian motion, independent from $\Fs_{t-\epsilon}$,
with starting point $X_{t-\epsilon}$ and covariance matrix
$\sigma(X_{t-\epsilon})\sigma^\star(X_{t-\epsilon})$.
By a discrete integration by parts (that is by using the
second formula in \eqref{e:discrete_increments}
and a change of variables),
\begin{equation}\label{e:bulkpe}
  \begin{aligned}
    \E[\Delta_h^m\phi(y+\sigma(y)\tilde B_\epsilon)]
      &= \int_{\R^d}\phi(y+x)
        \Delta_{-h}^m g_{\sigma(y)}(\epsilon,x)\,dx\\
      &\leq \|\phi\|_{L^\infty}
        \|\Delta_{-h}^m g_{\sigma(y)}(\epsilon)\|_{L^1},
  \end{aligned}
\end{equation}
where $g_{\sigma(y)}(\epsilon)$ is the density of
a Brownian motion with covariance matrix
$\sigma(y)\sigma^\star(y)$ at time $\epsilon$.
It is elementary now to show that
$\|\Delta_{-h}^m g_{\sigma(y)}(\epsilon)\|_{L^1}
\leq c(1\wedge(|h|/\sqrt{\epsilon}))^m$, with a number
$c$ that depends on $\mathop{det}(\sigma(y)\sigma^\star(y))^{-1}$
(and $\|\sigma\|_{L^\infty}$), and thus is uniformly
bounded with respect to $y$ by our non--degeneracy assumption.
In conclusion,
$\prober\lesssim\|\phi\|_{L^\infty}(|h|/\sqrt{\epsilon})^m$,
and
\[
  \E[\Delta_h^m\phi(X_t)]
    \lesssim \|\phi\|_{\Cs^\alpha_\beta}
      \bigl(\epsilon^{\frac12\alpha(1+\beta)}
      + \bigl(1\wedge\tfrac{|h|}{\sqrt{\epsilon}}\bigr)^m\bigr).
\]
An optimization in $\epsilon$ suggests the choice
$\epsilon\approx|h|^{\frac{2m}{m+\alpha(1+\beta)}}$
(we will address the issue that $\epsilon\leq t$ below),
thus
\begin{equation}\label{e:bulk}
  \E[\Delta_h^m\phi(X_t)]
    \lesssim \|\phi\|_{\Cs^\alpha_b}
      |h|^{\frac{\alpha m(1+\beta)}{m+\alpha(1+\beta)}}.
\end{equation}
For $m$ large, the exponent of $|h|$ is about $\alpha(1+\beta)$.
The smoothing Lemma~\ref{l:smoothing1} yields that
$X_t$ has a density $p_t\in B^{\alpha\beta}_{1,\infty}(\R^d)$.
Since $\alpha$ can be chosen arbitrarily close to $1$, we
conclude that $p_t\in B^{\beta-\eta}_{1,\infty}(\R^d)$
for all $\eta>0$.
\begin{remark}
  The regularity conditions on the coefficients
  can be replaced by growth conditions and finiteness
  of moments. For instance, assume that
  \[
    \E\bigl[\sup_{s\in[0,t]}|X_s|^2\bigr]
      <\infty,
  \]
  and that $|b(x)|\lesssim |x|^2$ (this is indeed the
  case analyzed in \cite{DebRom2014}). Then
  \[
    \E\Bigl|\int_0^t b(X_s)\,ds\Bigr|
      \lesssim \epsilon\E\bigl[\sup_{s\in[0,t]}|X_s|^2\bigr],
  \]
  and the estimate of the \Ae
  is as above.
\end{remark}
We wish to give a version of the above computations
that takes into account the size of the pre--factor
in \eqref{e:bulk} in terms of the time $t$ and the
initial value $X_0$. We state the result in a way
that will be convenient in the next sections.
In terms of the proposition below, the above
computations correspond to $a_0=\beta$, $\theta=2$.
\begin{proposition}\label{p:bulk}
  Consider a solution $X$ of \eqref{e:toy} with
  initial condition $X_0=x\in\R^d$.
  If there are numbers $a_0>0$, $\theta>0$,
  $K_0\geq1$ such that
  \begin{equation}\label{e:bulkstart}
    \E[\Delta_h^m\phi(X_t)]
      \leq K_0^{\frac\alpha\theta}\bigl(\epsilon^{\frac\alpha\theta(1+a_0)}
        + \bigl(\tfrac{|h|}{\epsilon^{1/\theta}}\bigr)^m\bigr)
        \|\phi\|_{\Cs^\alpha_b},
  \end{equation}
  for every $\epsilon<1$, with $\epsilon\leq\frac{t}2$,
  every $\alpha\in(0,1)$, and every
  $\phi\in \Cs^\alpha_b$, then
  \begin{equation}\label{e:bulkend}
    \|p_x(t)\|_{B^a_{1,\infty}}
      \leq c K_0^{\frac{a}{\theta a_0}+\delta}
        (1\wedge t)^{-\frac{1+a_0}{\theta a_0}a-\delta},
  \end{equation}
  for every $a\in(0,a_0)$ and small $\delta>0$, where
  $c=c(a,a_0,\theta,\delta)$ does not depend from $x$.
  Here $p_x(t)$ is the density of $X_t$.
\end{proposition}
\begin{proof}
  Indeed, if $a<a_0$, choose $\alpha=\tfrac{a+\delta_1}{a_0}$,
  with $\delta_1<a_0-a$, so that $\alpha<1$. Consider also
  $\delta_2<1$, to be chosen later, and choose $m\geq 1$
  as the smallest integer such that
  \[
    \frac{m}{m+\alpha(1+a_0)}
      \geq 1-\delta_2.
  \]

  Consider two cases. If $|h|^{\theta(1-\delta_2)}<t$, then
  we choose $\epsilon=\frac12|h|^{\theta(1-\delta_2)}$ so that
  \[
    \E[\Delta_h^m\phi(X_t)]
      \lesssim K_0^{\frac\alpha\theta}(|h|^{\alpha(1+a_0)(1-\delta_2)}
        + |h|^{m\delta_2})\|\phi\|_{\Cs_b^\alpha}
      \lesssim K_0^{\frac\alpha\theta}|h|^{\alpha(1+a_0)(1-\delta_2)}
        \|\phi\|_{\Cs_b^\alpha},
  \]
  since $\delta_2m\geq\alpha(1+a_0)(1-\delta_2)$ by the
  choice of $m$.
  
  If on the other hand $t\leq |h|^{\theta(1-\delta_2)}$,
  then we choose $\epsilon=\frac t2$, so that
  \[
    \begin{aligned}
      \E[\Delta_h^m\phi(X_t)]
        &\lesssim K_0^{\frac\alpha\theta}\bigl(t^{\frac\alpha\theta(1+a_0)}
          + \bigl(\tfrac{|h|}{t^{1/\theta}}\bigr)
            ^{\alpha(1+a_0)(1-\delta_2)}\bigr)\|\phi\|_{\Cs_b^\alpha}\\
        &\lesssim K_0^{\frac\alpha\theta}\bigl(
          (1\wedge t)^{-\frac\alpha\theta(1+a_0)(1-\delta_2)}
          |h|^{\alpha(1+a_0)(1-\delta_2)}\bigr)\|\phi\|_{\Cs_b^\alpha}.
    \end{aligned}
  \]
  
  In either case,
  \[
    \E[\Delta_h^m\phi(X_t)]
      \lesssim K_0^{\frac\alpha\theta}
        (1\wedge t)^{-\frac\alpha\theta(1+a_0)(1-\delta_2)}
        |h|^{\alpha(1+a_0)(1-\delta_2)}\|\phi\|_{\Cs_b^\alpha}.
  \]
  If we choose $\delta_2=\tfrac{\delta_1}{2\alpha(1+a_0)}$,
  then $\alpha(1+a_0)(1-\delta_2)-\alpha=a+\tfrac{\delta_1}2$,
  thus $a<\alpha(1+a_0)(1-\delta_2)-\alpha$. Since by our
  choice of $m$ we also have
  $m>\alpha(1+a_0)(1-\delta_2)$, by the smoothing
  Lemma~\ref{l:smoothing1},
  \[
    \|p_x(t)\|_{B^a_{1,\infty}}
      \lesssim K_0^{\frac\alpha\theta}
        (1\wedge t)^{-\frac\alpha\theta(1+a_0)(1-\delta_2)}.
  \]
  By the choice of $\delta_2$,
  \[
    \frac\alpha\theta(1+a_0)(1-\delta_2)
      = \frac{1+a_0}{\theta a_0}a
        + \frac{2+a_0}{2\theta a_0}\delta_1
      = \frac{1+a_0}{\theta a_0}a + \delta,
  \]
  if we set $\delta_1=\tfrac{2\theta a_0}{2+a_0}\delta$.
  With these positions,
  \[
    \frac\alpha\theta
      = \frac{a}{\theta a_0} + \frac{\delta_1}{\theta a_0}
      = \frac{a}{\theta a_0} + \frac2{2+a_0}\delta
      \leq \frac{a}{\theta a_0} + \delta.
  \]
  This yields the exponents in \eqref{e:bulkend}.
  
  Finally, $\delta_1<a_0-a$ if $\delta<\tfrac1{2\theta a_0}
  (a_0-a)(2+a_0)$, while with our positions,
  \[
    \delta_2
      = \frac{\theta a_0\delta}
        {(1+a_0)(2+a_0)\frac{a}{a_0}+2\theta(1+a_0)\delta}
      \leq\frac12.\qedhere
  \]
\end{proof}
\begin{remark}[Time dependent coefficients]
  It is not difficult to add time regularity to the
  coefficients. We will do so for the drift
  in Section~\ref{s:rougherdrift}. As for the
  diffusion coefficient, if for instance
  $\sigma\in L^\infty(0,T;C^\beta_b)$,
  and if we define the auxiliary process as
  \[
    Y^\epsilon_s
      = X_{t-\epsilon}
        + \sigma(s,X_{t-\epsilon})(B_s - B_{t-\epsilon}),
  \]
  if $s\geq t-\epsilon$, and as in \eqref{e:auxiliary}
  otherwise, then both the \Ae
  and the \Pe can be estimated
  with the same power of $\epsilon$ as above, and thus
  the result we get is the same.
\end{remark}
\begin{remark}[Time regularity]
  It is possible to obtain regularity in time of
  the density. This has been done in \cite{Rom2016b}
  in the framework of finite dimensional projections
  of Navier--Stokes equations. Here one can prove that
  \[
    \|p_x(t) - p_x(s)\|_{B^\alpha_{1,\infty}}
      \lesssim |t-s|^{\beta/2},
  \]
  when $\alpha+\beta<1$, where $p_x$ is the density
  of the solution started at $x$. The estimate of the
  semi--norm $[\cdot]_{B^\alpha_{1,\infty}}$
  (see \eqref{e:bseminorm} for its definition)
  is easy, since it is elementary to see that
  \[
    \|\Delta_h^n(p_x(t)-p_x(s))\|_{L^1}\lesssim
      \begin{cases}
        \|\Delta_h^n p_x(t)\|_{L^1} + \|\Delta_h^n p_x(s)\|_{L^1},
          &\qquad |h|\ll|t-s|,\\
        \|p_x(t) - p_x(s)\|_{L^1},
          &\qquad |t-s|\ll|h|.
      \end{cases}
  \]
  From this, using the methods we have introduced in this
  section, it is easy to derive that
  $[p_x(t)-p_x(s)]_{B^\alpha_{1,\infty}}\lesssim |t-s|^{\beta/2}$.
  The $L^1$--estimate
  $\|p_x(t) - p_x(s)\|_{L^1}\lesssim |t-s|^{\frac12-}$
  is the most challenging, and in \cite{Rom2016b} it has
  been obtained using the Girsanov transformation. 
\end{remark}
\subsection{Local solutions}\label{s:blowup}

In principle the solution of~\eqref{e:toy}
may not be defined over all times.
In this section we want to manage
the case of solutions with
explosion. We will see that,
under suitable assumptions
on the coefficients, the
solution of \eqref{e:toy} has a
density on the event of non-explosion.

To this end we need to give a suitable
definition of the solution.
We define $X_t$ to be the local
solution before explosion, and
a cemetery site $\infty$ afterwards.
Under the assumption of local
existence and uniqueness this
defines a Markov process (this was
done already for instance in
\cite{Rom2014} in a general
framework). Denote by $\tau_\infty$
the explosion time, then
it is easy to see that
$\{\tau_\infty>t\}=\{|X_t|<\infty\}$.
This will allow to localize at
fixed time the blow--up (unlike
$\tau_\infty$ that in principle
carries the full past history
of the process).

Let $\eta_R(x)$ be a smooth function
equal to $1$ when $|x|\leq R$,
and to $0$ when $|x|\geq R+1$.
Our computations will prove
that, if $\mu_t$ is the law of
$X_t$ (possibly with an atom
at $\infty$), then $\eta_R(x)\mu_t(dx)$
has a density with respect to
the Lebesgue measure. As long as the
estimate on the density (due to
the coefficients) of $\eta_R\mu_t$
does not depend from $R$, by
the uniform integrability ensured
by the Besov bound also the
measure $\uno_{\{|X_t|<\infty\}}\mu_t
=\uno_{\{\tau_\infty>t\}}\mu_t$
has a density, with similar
Besov bounds by semi--continuity.

In more details, fix $t>0$,
$\alpha\in(0,1)$, $\phi\in\Cs_b^\alpha$,
$m\geq1$, and $h$ such that $|h|\leq 1$.
To prove that $\eta_R(x)\mu_t(dx)$
has a density, in view of the
smoothing Lemma~\ref{l:smoothing1}
it is sufficient to prove that
\[
  \E[\Delta_h^m\phi(X_t)\eta_R(X_t)]
    \lesssim |h|^s\|\phi\|_{\Cs_b^\alpha},
\]
for some $s>\alpha$.
We decompose the discrete derivative as
\[
  \begin{multlined}[.9\linewidth]
    \E[\Delta_h^m\phi(X_t)\eta_R(X_t)]
      = \E\bigl[\bigl(\Delta_h^m\phi(X_t)
        -\Delta_h^m\phi(Y^\epsilon_t)\bigr)\eta_R(X_t)\bigr]\\
      + \E[\Delta_h^m\phi(Y^\epsilon_t)
        (\eta_R(X_t)-\eta_R(X_{t-\epsilon}))
      + \E[\Delta_h^m\phi(Y_t^\epsilon)\eta_R(X_{t-\epsilon})]
  \end{multlined}
\]
The first term on the right-hand-side is the \Ae,
\[
  \E\bigl[\bigl(\Delta_h^m\phi(X_t)
      -\Delta_h^m\phi(Y^\epsilon_t)\bigr)\eta_R(X_t)\bigr]
    \lesssim [\phi]_{\Cs^\alpha_b}
      \E[|X_t-Y_t^\epsilon|^\alpha\eta_R(X_t)],
\]
and the final estimate depends on the coefficients.
The third term is the \Pe,
\[
  \begin{aligned}
    \E[\Delta_h^m\phi(Y_t^\epsilon)\eta_R(X_{t-\epsilon})]
      &= \E\bigl[\eta_R(X_{t-\epsilon})
        \E[\Delta_h^m\phi(Y_t^\epsilon)\,|\,\Fs_{t-\epsilon}]\bigr]\\
      &\leq\E\bigl[|\E[\Delta_h^m\phi(Y_t^\epsilon)
        \,|\,\Fs_{t-\epsilon}]|\bigr]
  \end{aligned}
\]
that can be estimated as in the previous section.
Here $\Fs_{t-\epsilon}$ is the $\sigma$--field
of events before time $t-\epsilon$.

Finally, the new term that accounts for explosion
can be controlled as follows,
\[
  \E[\Delta_h^m\phi(Y^\epsilon_t)
      (\eta_R(X_t)-\eta_R(X_{t-\epsilon}))]
    \lesssim [\phi]_{\Cs^\alpha_b}|h|^\alpha
      \E[|\eta_R(X_t)-\eta_R(X_{t-\epsilon})|].
\]
Moreover, we have that
\[
  |\eta_R(X_t)-\eta_R(X_{t-\epsilon})|
    \leq |X_t-X_{t-\epsilon}|\eta_{R+1}(X_t),
\]
if both $X_t$ and $X_{t-\epsilon}$ are smaller
than $R+1$ or larger than $R+1$. In case
$X_{t-\epsilon}<R+1$ and $X_t\geq R+1$,
\[
  \begin{multlined}[.9\linewidth]
    |\eta_R(X_t)-\eta_R(X_{t-\epsilon})|
      = |\eta_R(X_{t-\epsilon})\eta_{R+1}(X_{t-\epsilon})|\\
      = |\eta_R(X_{\tau_{R+1}})-\eta_R(X_{t-\epsilon})|
        \eta_{R+1}(X_{t-\epsilon})
      \leq |X_{\tau_{R+1}}-X_{t-\epsilon}|\eta_{R+1}(X_{t-\epsilon}),
  \end{multlined}
\]
and in both cases one can use the equation to obtain
an estimate of the above quantities.
\subsection{More regularity - I}\label{s:morereg1}

By Proposition~\ref{p:bulk} it is immediately clear
that in order to obtain more regularity for the
density we need to improve the \Ae.
If we think of the auxiliary process~\ref{e:auxiliary}
as a (very) basic numerical approximation of the
original process $X$, then we need to use a more
refined numerical method. To do this it is necessary
to have a smoother coefficient. Moreover, due to the
non--anticipative nature of the estimate of the
\Pe, the numerical method
needs to be explicit.

Let us focus, for the sake of clarity, on the drift
term, namely consider
\[
  dX_t
    = b(X_t)\,dt + dB_t,
\]
with $b\in C^\beta_b$.
Similar considerations and conclusions are likewise
possible for the diffusion coefficient.

To exploit the additional regularity of
the drift $b$, it
is meaningful to define the auxiliary process
differently.
Define the auxiliary process $Y^\epsilon$
in a general way as
\begin{equation}\label{e:generalaux}
  Y^\epsilon_s =
    X_{t-\epsilon} + \int_{t-\epsilon}^s A_r\,dr
      + B_s - B_{t-\epsilon},
        \qquad s\geq t-\epsilon,
\end{equation}
and this time take $A_r=b(X_{t-\epsilon})$, rather than
$A_r=0$ as we have done before.
The \Pe
does not change, since we have only changed the mean, but
in a way that is measurable with respect to the information
at time $t-\epsilon$. To evaluate the \Ae,
consider
\begin{equation}\label{e:morereg1}
  \E[|X_t - Y_t^\epsilon|]
    = \E\Bigl|\int_{t-\epsilon}^t(b(X_s)-b(X_{t-\epsilon}))\,ds\Bigr|
    \leq [b]_{C^\beta}\E\int_{t-\epsilon}^t
      |X_s-X_{t-\epsilon}|^\beta\,ds
    \lesssim \epsilon^{1+\frac{\beta}2},
\end{equation}
since, as in Section~\ref{s:bulk},
$\E[|X_s-X_{t-\epsilon}|]\lesssim\sqrt\epsilon$. Thus,
$\apperr\lesssim [\phi]_{\Cs^\alpha_b}\epsilon
^{\frac\alpha2(2+\beta)}$ and Proposition~\ref{p:bulk}
ensures that the density is in $B^a_{1,\infty}$
for $a<1+\beta$.

It is now quantitatively clear that the more regular
is $b$, the more regularity we obtain for the
density. The (obvious) key idea is to find
a good approximation of $X$. 
A natural candidate then for $Y$
are the Picard iterations for our equation.
The next step to improve the regularity of the density
is to choose an auxiliary process that ensures
a smaller estimate of $|X_t-Y_t^\epsilon|$. 
Since  in \eqref{e:morereg1} the estimates depends on
$|X_s-X_{t-\epsilon}|$, and in turns the size of
this term corresponds to the size of the Brownian
increments, a way to improve the difference might be
to define the auxiliary process \eqref{e:generalaux}
with $A_s=b(X_{t-\epsilon}+B_s-B_{t-\epsilon})$
when $s\geq t-\epsilon$.
Unfortunately this makes the term in the
\Pe ``anticipative'' (with respect
to the time $t-\epsilon$). If we are ready to consider
an ``anticipative'' term, we are faced with the difficulty
that it becomes now difficult to evaluate the law of the
terms that contribute to the \Pe.
This is indeed an implicit requirement of the simple method
we are illustrating.

A workaround that, albeit ``anticipative'', keeps the
term that will contribute to the \Pe
simple, can be considered if $b$ is more regular,
namely if $b\in C^{1+\beta}_b(\R^d)$, by looking at
iterated integrals of Brownian motion. Additional
regularity of $b$ cannot be avoided in general,
and a way to see this is to look at the discussion
on the Fokker--Planck equation at the beginning of
Section~\ref{s:morereg2}.
With this in mind, take
\begin{equation}\label{e:aux2}
  A_r
    = b(X_{t-\epsilon}) +b'(X_{t-\epsilon})(B_r-B_{t-\epsilon})
\end{equation}
in \eqref{e:generalaux}, then
\[
  \begin{aligned}
    b(X_s) - A_s
      &= b(X_s) - b(X_{t-\epsilon})
        - b'(X_{t-\epsilon})(B_r-B_{t-\epsilon})\\
      &\lesssim b'(X_{t-\epsilon})(X_s-X_{t-\epsilon})
        + O(|X_s-X_{t-\epsilon}|^{1+\beta})
       - b'(X_{t-\epsilon})(B_r-B_{t-\epsilon})\\
      &\lesssim b'(X_{t-\epsilon})\int_{t-\epsilon}^t b(X_r)\,dr
        + O(|X_s-X_{t-\epsilon}|^{1+\beta}),
  \end{aligned}
\]
thus $\E[|b(X_s) - A_s|]\lesssim \epsilon^{\frac12(1+\beta)}$,
therefore $\apperr\lesssim [\phi]_{\Cs^\alpha_b}
\epsilon^{\frac\alpha2(3+\beta)}$.
With the choice of the auxiliary process, we see that the
term providing the \Pe
is again amenable to our analysis: given
$y=X_{t-\epsilon}$, we have that
\[
  Y^\epsilon_t
    = y+\epsilon b(y) + \tilde B_\epsilon
      + b'(y)\int_0^\epsilon\tilde B_s\,ds,
\]
where $\tilde B_r=B_{t-\epsilon+r}-B_{t-\epsilon}$ is
independent from the history until time $t-\epsilon$.
The random variable $Y^\epsilon_t$ is conditionally
Gaussian with variance
\[
  \Var(Y_t^\epsilon|X_{t-\epsilon})
    = \epsilon + b'(X_{t-\epsilon})\epsilon^2
      +\frac13b'(X_{t-\epsilon})^2\epsilon^3
    \gtrsim\epsilon,
\]
so that the \Pe has the same
order in $\epsilon$.
In conclusion the density is in $B^a_{1,\infty}$
for all $a<\beta+2$.

With additional regularity of $b$ one can consider
better stochastic Taylor approximations to obtain
better estimates. We refer to the classical
\cite[Chapter 10]{KloPla1992} for some possibilities.

As a concluding remark we notice that, as long as
the coefficients have derivatives, one can more easily
resort to Malliavin calculus. These consideration will
become useful though in Section~\ref{s:hypo}.
\section{Singular diffusion coefficient}\label{s:singular}

In this section we wish to consider the case when
the diffusion matrix is singular. We will briefly
discuss two cases. The first is when
the diffusion matrix is non-invertible.
The second case is the so--called
hypo-elliptic case, when the diffusion
matrix is singular, but the noise
is transmitted to noise-less components
by the drift.
\subsection{Singular diffusion coefficient}\label{s:nonelliptic}

First of all, we notice that in general we
do not need a uniform estimate
such as \eqref{e:bulknonsingular}, since
under the assumption of non--singularity
one can use the results of Section~\ref{s:local}.
In this section we wish to investigate if there
is any condition that can ensure the existence
of a density even around point where the
diffusion coefficient is zero or non--invertible.

An observation from \cite{DebFou2013} indeed
says that the method from Section~\ref{s:bulk}
can be slightly modified to take into account
the case when the diffusion coefficient is not
invertible. The idea is as in Section~\ref{s:blowup}.
Consider a solution $X_t$ of \eqref{e:toy} and let
$\mu_t$ be its law. Fix an integer $m\geq1$,
then the smoothing Lemma~\ref{l:smoothing1}
will be applied to the measure
$|\sigma^{-1}(x)|^{-m}\mu_t(dx)$, where $|\cdot|$
is the operator norm of a matrix. If
this measure has a density, then $\mu_t$
has a density on $\{y:\sigma(y)\text{ invertible}\}$.

To this end, it is sufficient to prove that
\[
  \E\Bigl[\frac{\Delta^m_h\phi(X_t)}{|\sigma^{-1}(X_t)|^m}\Bigr]
    \lesssim |h|^s\|\phi\|_{\Cs^\alpha_b}.
\]
The decomposition here is
\[
  \begin{aligned}
    \E[|\sigma^{-1}(X_t)|^{-m}\Delta^m_h\phi(X_t)]
      &= \E[\Delta^m_h\phi(X_t)(|\sigma^{-1}(X_t)|^{-m}
        -|\sigma^{-1}(X_{t-\epsilon})|^{-m})] + {}\\
      &\quad + \E[|\sigma^{-1}(X_{t-\epsilon})|^{-m}
        (\Delta^m_h\phi(X_t)-\Delta^m_h\phi(Y_t^\epsilon))]
        + {}\\
      &\quad + \E[|\sigma^{-1}(X_{t-\epsilon})|^{-m}
        \Delta^m_h\phi(Y_t^\epsilon)]
  \end{aligned}
\]
The first term can be controlled as
\[
  \begin{aligned}
    \E[\Delta^m_h\phi(X_t)(|\sigma^{-1}(X_t)|^{-m}
        -|\sigma^{-1}(X_{t-\epsilon})|^{-m})]
      &\lesssim [\phi]_{\Cs^\alpha_b}|h|^\alpha
        [\sigma]_{C^\beta}\E[|X_t - X_{t-\epsilon}|^\beta]\\
      &\lesssim [\phi]_{\Cs^\alpha_b}|h|^\alpha
        \epsilon^{\frac\beta2},
  \end{aligned}
\]
while the second term plays the role
of the \Ae, thus
\[
  \apperr
    \lesssim \|\sigma\|_{L^\infty}^m[\phi]_{\Cs^\alpha_b}
      \E[|X_t - Y_t^\epsilon|^\alpha]
    \lesssim [\phi]_{\Cs^\alpha_b}
      \epsilon^{\frac\alpha2(1+\beta)}.
\]
Finally, the third term is the \Pe,
and by conditioning,
\[
  \prober
    = \E\bigl[\bigl(
      |\sigma^{-1}(y)|^{-m}
      \E[\Delta_h^m\phi(y+\sigma(y)\tilde B_\epsilon)]
      \bigr)_{y=X_{t-\epsilon}}\bigr]
    \lesssim \|\phi\|_{L^\infty}\epsilon^{-\frac{m}2}|h|^m.
\]
The contribution of the first term is negligible,
since, if the method is successful, then
$|h|\lesssim\sqrt\epsilon$, thus
$|h|^\alpha\epsilon^{\beta/2}
\lesssim\epsilon^{\frac12(\alpha+\beta)}
\leq\epsilon^{\frac\alpha2(1+\beta)}$,
since $\alpha<1$. In conclusion, we end up with
the same estimate as in Section~\ref{s:bulk},
\[
  \E\Bigl[\frac{\Delta^m_h\phi(X_t)}{|\sigma^{-1}(X_t)|^m}\Bigr]
    \lesssim \|\phi\|_{\Cs^\alpha_b}
      (\epsilon^{\frac\alpha2(1+\beta)}
      + \epsilon^{-\frac{m}2}|h|^m),
\]
so that we obtain the same conclusion,
but on the measure $|\sigma^{-1}(x)|^{-m}\mu_t(dx)$.
This proves the following result.
\begin{proposition}
  Let $b\in L^\infty(\R^d)$ and
  $\sigma\in C^\beta_b(\R^d:\R^{d\times d'})$.
  Let $X$ be a solution of \eqref{e:toy} and
  $t>0$. Then $X_t$ has a density with respect
  to the Lebesgue measure on the set
  $\{y:\sigma(y)\text{ invertible}\}$.
\end{proposition}
Clearly, the regularity assumptions on the
coefficients can be relaxed (for instance
along the lines of Section~\ref{s:rougher}).

We now wish to find some simple condition
for the existence of a density on $\R^d$
even in the case of non--invertible
diffusion coefficient.
For the sake of simplicity, we consider
our problem \eqref{e:toy} without drift,
namely,
\[
  dX_t
    = \sigma(X_t)\,dB_t.
\]
\begin{proposition}
  Assume that $\sigma\in C^\beta_b(\R^d;\R^{d\times d})$,
  let $t>0$ and assume that there is $\gamma>0$ such that
  \[
    \E\Bigl[\int_{\frac{t}2}^t|\sigma(X_s)^{-1}|^\gamma
        \,ds\Bigr]
      <\infty.
  \]
  Then $X_t$ has a density $p$ in $\R^d$
  and there is $a_0=a_0(\beta,\gamma)>0$ such
  that $p\in B^a_{1,\infty}$ for all $a<a_0$.
\end{proposition}
\begin{proof}
  Consider the auxiliary process as in~\eqref{e:auxiliary}.
  The \Ae is as in Section~\ref{s:bulk},
  thus $\apperr\lesssim \epsilon^{\frac\alpha2(1+\beta)}
  [\phi]_{\Cs^\alpha_b}$.

  For the \Pe, if
  $\tilde B_s=B_{t-\epsilon+s}-B_{t-\epsilon}$, then
  by conditioning,
  \[
    \prober
      = \E\bigl[\E[\Delta_h^m\phi(y+\sigma(y)
        \tilde B_\epsilon)]_{y=X_{t-\epsilon}}\bigr]
  \]
  Now, if $\sigma(y)$ is invertible,
  \[
    \E[\Delta_h^m\phi(y+\sigma(y)\tilde B_\epsilon)]
      \lesssim \Bigl(1\wedge\frac{|\sigma(y)^{-1}|\,|h|}
        {\sqrt\epsilon}\Bigr)^m\|\phi\|_{L^\infty},
  \]
  and the inequality extends to $\sigma(y)$ non invertible,
  since by definition $|\sigma(y)^{-1}|^{-1}
  = \inf_{|z|=1}|\sigma(y)z|$ (and the understanding
  $\frac10=\infty$). By the moment assumption,
  \[
    \prober
      \lesssim\E\Bigl[\Bigl(1\wedge
        \frac{|\sigma(X_{t-\epsilon})^{-1}|\,|h|}
        {\sqrt\epsilon}\Bigr)^m\Bigr]\|\phi\|_{L^\infty}
      \lesssim \E[|\sigma(X_{t-\epsilon})^{-1}|^\gamma]
        \Bigl(\frac{|h|}{\sqrt\epsilon}\Bigr)^\gamma
        \|\phi\|_{L^\infty}.
  \]
  Choose now $0<\lambda_1<\lambda_2<1$ and consider
  the two terms \apperr and \prober (here we have
  replaced $\epsilon$ with $\delta$ for convenience),
  \[
    \E[\Delta_h^m\phi(X_t)]
      \lesssim \Bigl(\delta^{\frac\alpha2(1+\beta)}
        + \frac{|h|^\gamma}{\delta^{\gamma/2}}
        \E[|\sigma(X_{t-\delta})^{-1}|^\gamma]
        \Bigr)\|\phi\|_{\Cs^\alpha_b}.
  \]
  Integrate the above inequality over
  $\delta\in(\lambda_1\epsilon,\lambda_2\epsilon)$
  to obtain,
  \[
    \E[\Delta_h^m\phi(X_t)]
      \lesssim \Bigl(\epsilon^{\frac\alpha2(1+\beta)}
        + \frac{|h|^\gamma}{\epsilon^{\gamma/2}}
        \E\int_{\frac{t}2}^t|\sigma(X_r)^{-1}|^\gamma\,dr
        \Bigr)\|\phi\|_{\Cs^\alpha_b}.
  \]
  The proof can now be concluded with computations
  similar to those in Proposition~\ref{p:bulk}.
\end{proof}
\begin{example}
  The result we have obtained is clearly
  non--optimal. Let us consider a simple
  example to figure out where the issues
  arise.
  
  Consider $X_t=B_t^2$, where $(B_t)_{t\geq0}$
  is a one--dimensional Brownian motion. It is
  known that $(X_t)_{t\geq0}$ solves
  \[
    dX_t
      = dt + 2\sqrt{X_t}\,dW_t,
  \]
  with initial condition $X_0=0$, where
  $(W_t)_{t\geq0}$ is another Brownian
  motion.
  The function $\sigma(x)=\sqrt{x}$ is clearly
  $\frac12$--H\"older, and
  \[
    \E[\sigma(X_t)^\gamma]
      = \E[|B_t|^{-\gamma}]
      <\infty
  \]
  if and only if $\gamma<1$. The previous
  considerations ensure that there is a
  density $p_t$ for $X_t$ on $\R$ and that has
  regularity $B_{1,\infty}^\alpha$, with
  $\alpha<\frac23(\frac52-2\sqrt{\frac32})\ll\frac12$.
  On the other hand we know explicitly the
  density, $p_t(x)=(2\pi t x)^{-\frac12}\e^{-x/(2t)}$,
  thus it is not difficult to prove that actually
  $p_t\in B_{1,\infty}^\alpha$, $\alpha\leq\frac12$.
  The correct regularity can be recovered, at least away
  from the zero of the diffusion coefficient, using the
  methods of Section~\ref{s:local}.
\end{example}
\subsection{A hypo-elliptic example}\label{s:hypo}

Here we aim to consider a hypo-elliptic problem, that
is a problem where the diffusion coefficient is not
invertible, but the effect of the noise is
propagated by the drift. We will only show a
(very) elementary example, to convince that
the simple method we are illustrating is
effective also in this framework. It would
be difficult, though, to give a general result.

Consider the following problem,
with $X=(X^1,X^2)$,
\begin{equation}\label{e:multitoy}
  \begin{cases}
    dX^1
      = b_1(X^1,X^2)\,dt + dB,\\
    dX^2
      = b_2(X^1,X^2)\,dt,
  \end{cases}
\end{equation}
where $B$ is a one--dimensional standard
Brownian motion. Here the diffusion
coefficient is
\[
  \begin{pmatrix}
    1 & 0\\
    0 & 0
  \end{pmatrix},
\]
which is nowhere invertible, thus none of the
results of the previous section is
available. The key point is clearly the
definition of the auxiliary process for
the second component. The basic definition
in \eqref{e:auxiliary}, as well as the
one in \eqref{e:generalaux} with
$A^2_r=b_2(X_{t-\epsilon})$, are not
suitable, because we do not introduce
any influence of the noise. This is
fundamental in view of the \Pe.
Thus, we consider the
auxiliary process in the second component
as in \eqref{e:generalaux}, but with
the process $A$ defined similarly to \eqref{e:aux2},
and only with variations in the first component
(the one that contains the noise). In other
words,
\[
  Y_s^{\epsilon,2}
    = X_{t-s}^2 + \int_{t-\epsilon}^s\bigl(b_2(X_{t-\epsilon})
      +\partial_{x_1}b_2(X_{t-\epsilon})(Y_r^1-X_{t-\epsilon}^1)
      \bigr)\,dr,
\]
for $s\geq t-\epsilon$.
In particular, we need to assume that
$b_2$ is differentiable. This is not
sufficient, since to ensure that the
first component will ``transfer'' the 
regularizing effect of the noise to
the second component, we need to
add the ``hypo-elliptic'' assumption
\[
  \partial_{x_1}b_2(x)
    \neq0.
\]
We turn to the definition of the
first component of the auxiliary
process. There is a hidden requirement
for the \Ae.
Indeed, we will see in Proposition~\ref{p:hypo}
below that $\prober\lesssim\epsilon^{-\frac32m}|h|^m$.
This is due to the fact that the random variable
in the second component is smoother, thus
gives a stronger singularity in time.
To guess the right size of the
\Ae, we see that,
with the \Pe
as above, if $\apperr\lesssim\epsilon^{\alpha q}$,
then by the simple optimization we have seen
in Proposition~\ref{p:bulk}, we need
$q>\frac32$. In other words, we need
to choose an auxiliary process that
approximates the original process to
the order $\epsilon^{3/2}$. Thus we
assume $b_1\in C^{1+\beta}_b(\R^d)$,
with $\beta>0$, and set for $s\geq t-\epsilon$,
\[
  Y_s^{\epsilon,1}
    = X_{t-\epsilon}^1
      + \int_{t-\epsilon}^s \bigl(b_1(X_{t-\epsilon})
      + \partial_{x_1} b_1(X_{t-\epsilon})
      (B_s-B_{t-\epsilon})\bigr)\,ds
      + B_s - B_{1-\epsilon}.
\]
\begin{proposition}\label{p:hypo}
  Assume that $b_1\in C^{1+\beta}_b(\R^d)$, with
  $\beta\in(0,1)$, and $b_2\in C^1_b(\R^d)$, with
  $|\partial_{x_1}b_2(x)|\geq c_0>0$.
  Then for every $t>0$ any solution $X_t$
  of problem~\eqref{e:multitoy} has a density
  in $B^a_{1,\infty}$
  for every $a<\frac13\beta$.
\end{proposition}
We notice that on the one hand the result is
not very satisfactory, since coefficients
are assumed differentiable (although one can
be more careful on which directional derivatives
are really necessary), thus one can derive
the existence of a density by the existence
of the Malliavin derivative and the
non--degeneracy of the Malliavin matrix.
The last statement would follow from our
hypo-elliptic assumption. On the other hand,
at this level the simple method of this paper
still provides the additional value of the
minimal regularity that ensures that
smooth approximations of the solution have
uniformly integrable (thus, weakly compact)
densities.
\begin{proof}[Proof of Proposition~\ref{p:hypo}]
We first estimate the \Ae.
We easily have
\[
  \begin{gathered}
    \E[|X^1_s - X^1_{t-\epsilon}|]
      \leq \E[|B_s - B_{t-\epsilon}|]
        + \E\int_{t-\epsilon}^s |b_1(X_r)|\,dr
      \lesssim\sqrt{\epsilon},\\
    \E[|X^2_s - X^2_{t-\epsilon}|]
      \leq \E\int_{t-\epsilon}^s |b_2(X_r)|\,dr
      \leq\|b\|_{L^\infty}\epsilon,
  \end{gathered}
\]
thus,
\[
  \begin{aligned}
    \E[|X^1_t - Y^1_t|]
      &=\E\Bigl|\int_{t-\epsilon}^t
        \bigl(b_1(X_s) - b_1(X_{t-\epsilon}) 
        - \partial{x_1}b_1(X_{t-\epsilon})
        (X_s - X_{t-\epsilon}) + {}\\
      &\qquad+ \partial{x_1}b_1(X_{t-\epsilon})
        (X_s - X_{t-\epsilon}
        - (B_s - B_{t-\epsilon}))\bigr)\,ds\\
      &\lesssim \epsilon^{\frac12(\beta+3)}.
  \end{aligned}
\]
Likewise,
\[
  \begin{aligned}
    \E[|X^2_t - Y^{\epsilon,2}_t|]
      &\leq \E\int_{t-\epsilon}^t
        |b_2(X_s) - b_2(X_{t-\epsilon})
        - \partial_{x_1}b_2(X_{t-\epsilon})
        (Y^1_s-X^1_{t-\epsilon})|\,ds\\
      &\leq \E\int_{t-\epsilon}^t\bigl(
        |\partial_{x_1}b_2(X_{t-\epsilon})(X^1_s-Y^1_s)
        + \partial_{x_2}b_2(X_{t-\epsilon})
        (X_s^2-X_{t-\epsilon}^2)\bigr)\,ds\\
      &\leq \|\nabla b\|_{L^\infty}\int_{t-\epsilon}^t
        \E[|X^1_s - Y_s^1| + |X_s^2 - X^2_{t-\epsilon}|]\,ds\\
      &\lesssim\epsilon^2.        
  \end{aligned}
\]
In conclusion the \Ae is
$\apperr\lesssim [\phi]_{\Cs^\alpha_b}\epsilon^{\frac\alpha2(\beta+3)}$.

We turn to the \Pe.
Conditional to the history up to time $t-\epsilon$,
we can write
\[
  Y^1_t
    = A_1 + \tilde B_\epsilon + A_2 \tilde C_\epsilon,
      \qquad
  Y^2_t
    = A_3 + A_4 \tilde C_\epsilon
      \qquad\text{with}\qquad
  \tilde C_\epsilon
    \eqdef \int_0^\epsilon \tilde B_s\,ds,
\]
where $A_1$, $A_2$, $A_3$, $A_4$ are
measurable with respect to the past
(of $t-\epsilon$), and
$\tilde B_s=B_{s+t-\epsilon}-B_{t-\epsilon}$.
In particular, $A_4=\partial_{x_1}b_2(X_{t-\epsilon})$,
thus $|A_4|\geq c_0>0$.
The random variable 
$(\tilde B_\epsilon,\tilde C_\epsilon)$
is Gaussian, and its covariance matrix has eigenvalues
of order $\epsilon$ and $\epsilon^3$. Thus
\[
  \E[\Delta_h\phi(Y_t)]
    \lesssim \frac{|h|^m}{\epsilon^{\frac32m}}
      [\phi]_{\Cs^\alpha_b},
\]
and in conclusion
\[
  \E[\Delta_h^m\phi(X_1)]
    \lesssim [\phi]_{\Cs^\alpha_b}\epsilon^{\frac12\alpha(3+\beta)}
      + \|\phi\|_{L^\infty} \epsilon^{-\frac32m}|h|^m.
\]
Our standard computations (those in Proposition~\ref{p:bulk}),
ensure that the density of $X_t$ is in $B^a_{1,\infty}$
for all $a<\frac13\beta$.
\end{proof}
\begin{remark}
  The regularity obtained in the previous
  proposition is the regularity
  of the joint density of $X_t^1$ and $X_t^2$.
  It is not difficult to check that
  the density of the random variable $X_t^1$
  is more regular.
\end{remark}
\section{Local estimates}\label{s:local}

In this section we wish to obtain local estimates on the
density. This might be useful if for instance
\begin{itemize}
  \item we only have local regularity of the coefficients,
  \item or if the diffusion coefficient is non--zero
    or non--singular only in some part of the space,
\end{itemize}
and so on.

To this end, let us consider our toy model \eqref{e:toy}.
We will localize the problem as in \cite[Theorem 2.4]{Dem2011},
and then apply the method to the localized problem.
\begin{theorem}\label{t:local}
  Assume that there is a ball $D\subset\R^d$
  such that $b\in L^\infty(D)$,
  $\sigma\in C^\beta(D)$, for some $\beta>0$
  and that $\mathop{det}(\sigma(y)\sigma(y)^\star)>0$
  on $D$.

  If $X^x$ is solution of \eqref{e:toy}, with initial
  condition $x\in\R^d$, then for every $t>0$ the random
  variable $X^x_t$ has a density $p_x(t)$ in a smaller
  ball $D'\subseteq D$, and $p_x(t)\in B^a_{1,\infty}(D')$
  for every $a<\beta$.
\end{theorem}
We will devote the rest of the section to the proof
of the theorem. Prior to this, we give a global version
of the above result.
\begin{corollary}
  Assume that $b\in L^\infty_\text{loc}(\R^d)$,
  that $\sigma$ is H\"older continuous on bounded sets,
  and that $\mathop{det}(\sigma(y)\sigma(y)^\star)>0$
  for all $y\in\R^d$.
  If $X^x$ is solution of \eqref{e:toy}, with initial
  condition $x\in\R^d$, then for every $t>0$ the
  random variable $X^x_t$ has a density $p_x(t)$
  on $\R^d$. Moreover, if $D\subset\R^d$ is a
  ball and $\sigma\in C^\beta(D)$, then
  $p_x(t)\in B^a_{1,\infty}(D')$ for every
  $a<\beta$ and every smaller ball $D'\subset D$. 
\end{corollary}
\begin{proof}
  Notice that a probability measure that
  is locally absolutely continuous, is absolutely
  continuous, thus it is sufficient to prove the existence
  of a density in balls. This is immediate from
  the previous theorem.
\end{proof}
\subsection{Proof of Theorem~\ref{t:local}}

Let $x_0\in\R^d$ and $r>0$ be such that
$B_{6r}(x_0)\subset D$. We shall study existence and
regularity of the density around $x_0$ of the solution $X$
of \eqref{e:toy} at some fixed time $t$.
\subsubsection{Localization}

Let $\varphi\in C^\infty(\R^d)$ be such that
$\varphi=1$ on $B_1(0)$ and
$\uno_{B_1(0)}\leq\varphi\leq\uno_{B_2(0)}$,
and let $\varphi_r(x)\eqdef\varphi((x-x_0)/r)$.
Set
\[
  m(t,r)=\E[\varphi_r(X_t)],
\]
then we see that $r\mapsto m(t,r)$ is non--decreasing
with limit $1$ at $r=\infty$. Assume that $m(t,r)>0$
for all $r>0$ (otherwise $|X_t-x_0|>0$ {a.\,s.} and the
density of $X_t$ is equal to $0$ in a neighbourhood of
$x_0$). By the definition of $\varphi$, if
$\Supp(f)\subset B_r(x_0)$, then
\[
  \E[f(X_t)]
    = \E[f(X_t)\varphi_r(X_t)],
\]
thus, in view of the smoothing Lemma~\ref{l:smoothing1},
it is sufficient to prove that
\begin{equation}\label{e:target}
  \E[(\Delta_h^m\phi)(X_t)\varphi_r(X_t)]
    \lesssim |h|^\gamma\|\phi\|_{\Cs_b^\alpha},
\end{equation}
for suitable $\alpha$, $\gamma$.

To localize the dynamics, we proceed again as in \cite{Dem2011},
and consider a localizing function $\eta\in C^\infty(\R^d,\R^d)$,
defined as
$\eta(x)=x$ in $B_{4r}(0)$, and $\eta(x)=5x/|x|$ outside
$B_{5r}(0)$. Set $\bar b(x)=b(\eta(x))$ and $\bar\sigma(x)
=\sigma(\eta(x))$. It is immediate to see that
$\bar b$ and $\bar\sigma$
have (globally) the same (local) regularity properties of
$b$, $\sigma$. In particular, $\bar\sigma\in C^\beta_b(\R^d)$.
Let $\bar X^x$ be the solution of
\[
  d\bar X
    = \bar b(\bar X_s)\,ds + \bar \sigma(\bar X_s)\,dB_s,
\]
with initial condition $\bar X^x_0=x$.

\subsubsection{Decomposition}

Fix a small parameter $\delta\in(0,1]$, with $\delta\ll t$,
and set
\[
  \begin{gathered}
    \tau_i
      \eqdef\{s\geq t-\delta: X_s\in B_{3r}(x_0)\},\\
    \tau_o
      \eqdef\{s\geq\tau_i: X_s\not\in B_{4r}(x_0)\}.
  \end{gathered}
\]
Set moreover
\[
  \begin{gathered}
    I_t
      \eqdef\{\varphi_r(X_t)>0,\tau_i=t-\delta,\tau_o>t\},\\
    O_t
      \eqdef\{\varphi_r(X_t)>0,(\tau_i>t-\delta
        \text{ or }(\tau_i=t-\delta\text{ and }\tau_o\leq t))\},
  \end{gathered}
\]
so that if $f\in C_b(\R^d)$,
\begin{equation}\label{e:dec}
  \E[f(X_t)\varphi_r(X_t)]
    = \E[f(X_t)\varphi_r(X_t)\uno_{I_t}]
      + \E[f(X_t)\varphi_r(X_t)\uno_{O_t}]
\end{equation}
Consider first the second term. In \cite{Dem2011} it is proved that
\[
  O_t
    \subset\bigl\{\varphi_r(X_t)>0,\sup_{[0,\delta]}|\bar X_s^{X_{\tau_i}}-X_{\tau_i}|\geq r|\bigr\},
\]
and it is easy to see (recall that $\bar X$ satisfies
an equation with globally bounded coefficients) that
for every $q\geq1$,
\[
  \begin{aligned}
    \Prob\bigl[\varphi_r(X_t)>0,\sup_{[0,\delta]}
        |\bar X_s^{X_{\tau_i}}-X_{\tau_i}|\geq r|\bigr]
      &\leq \Prob\bigl[\tau_i\leq t,\sup_{[0,\delta]}
        |\bar X_s^{X_{\tau_i}}-X_{\tau_i}|\geq r|\bigr]\\
      &\leq\frac1{r^q}\E\Bigl[\uno_{\tau_i\leq t}
        \E\bigl[\sup_{[0,\delta]}
        |\bar X_s^{X_{\tau_i}}-X_{\tau_i}|^q\,\big|\tau_i\bigr]\Bigr]\\
      &\leq\frac1{r^q}\sup_{x\in B_{3r}(x_0)}\E\bigl[
        \sup_{[0,\delta]}|\bar X_s^{x}-x|^q\bigr]\\
      &\lesssim\frac1{r^q}(\|\bar b\|_{L^\infty}^q +
        \|\bar\sigma\|_{L^\infty}^q)\delta^{\frac{q}2}.
  \end{aligned}
\]
Therefore, for every $q\geq1$,
\begin{equation}\label{e:decb}
  \E[f(X_t)\varphi_r(X_t)\uno_{O_t}]
    \leq\|f\|_{L^\infty}\|\varphi\|_{L^\infty}\Prob[O_t]
    \lesssim\frac{\|\varphi\|_{L^\infty}}{r^q}
      (\|\bar b\|_{L^\infty}^q + \|\bar\sigma\|_{L^\infty}^q)
      \|f\|_{L^\infty}\delta^{\frac{q}2}.
\end{equation}

Consider next the first term in \eqref{e:dec}.
Conditional to $X_{t-\delta}$, the random
variable $\uno_{\{\tau_i=t-\delta\}}$
is $X_{t-\delta}$--measurable,
(since $\{\tau_i=t-\delta\}
=\{X_{t-\delta}\in B_{3r}(x_0))$,
thus
\begin{equation}\label{e:decap1}
  \begin{multlined}[.9\linewidth]
    \E\bigl[f(X_t)\varphi_r(X_t)\uno_{\{\varphi_r(X_t)>0\}}
        \uno_{\{\tau_i=t-\delta\}}\uno_{\{\tau_o>t\}}\bigr] = \\
      = \E\bigl[\uno_{\{\tau_i=t-\delta\}}
        \E[f(X_t)\varphi_r(X_t)\uno_{\{\varphi_r(X_t)>0\}}
        \uno_{\{\tau_o>t\}}\,\bigl|X_{t-\delta}]\bigr].
   \end{multlined}
\end{equation}
On the event $I_t$ we have that $X_{t-\delta}\in B_{3r}(x_0)$
and $X_s\in B_{4r}(x_0)$ for all $s\in[t-\delta,t]$.
Therefore
$X_s=X_{s+\delta-t}^{X_{t-\delta}}$ for all $s\in[t-\delta,t]$.
Thus, by the Markov property, on $\{\tau_i=t-\delta\}$,
\begin{equation}\label{e:decap2}
  \begin{aligned}
    \E\bigl[f(X_t)\varphi_r(X_t)\uno_{\{\varphi_r(X_t)>0\}}
        \uno_{\{\tau_o>t\}}\bigl|X_{t-\delta}\bigr]
      &= \E\bigl[f(\bar X_t)\varphi_r(\bar X_t)
        \uno_{\{\varphi_r(\bar X_t)>0\}}
        \uno_{\{\tau_o>t\}}\bigl|X_{t-\delta}\bigr]\\
      &= \E\bigl[f(\bar X_\delta^x)\varphi_r(\bar X_\delta^x)
        \uno_{\{\varphi_r(\bar X_\delta^x)>0\}}
        \uno_{\{\bar\tau^x>\delta\}}\bigr]\Bigl|_{x=X_{t-\delta}}\\
      &\leq\sup_{x\in B_{3r}(x_0)}\Bigl|
        \E\bigl[f(\bar X_\delta^x)\varphi_r(\bar X_\delta^x)
        \uno_{\{\bar\tau^x>\delta\}}\bigr]\Bigr|
  \end{aligned}
\end{equation}
where $\bar\tau^x$ is the first exit time of $\bar X^x$
from $B_{4r}(x_0)$.
We consider now the above expectation. If $x\in B_{3r}(x_0)$,
\begin{equation}\label{e:deca}
  \begin{multlined}[.8\linewidth]
  \E\bigl[f(\bar X_\delta^x)\varphi_r(\bar X_\delta^x)
      \uno_{\{\varphi_r(\bar X_\delta^x)>0\}}
      \uno_{\{\bar\tau^x>\delta\}}\bigr] = \\
    = \E\bigl[f(\bar X_\delta^x)\varphi_r(\bar X_\delta^x)\bigr]
      - \E\bigl[f(\bar X_\delta^x)\varphi_r(\bar X_\delta^x)
          \uno_{\{\bar\tau^x\leq\delta\}}\bigr].
  \end{multlined}
\end{equation}
The analysis of the first term on the right-hand side is
postponed to the next section. We focus here on
the second term. Indeed,
\[
  \E\bigl[f(\bar X_\delta^x)\varphi_r(\bar X_\delta^x)
      \uno_{\{\bar\tau^x\leq\delta\}}\bigr]
    \leq \|f\|_{L^\infty}\|\varphi\|_{L^\infty}\Prob[\bar\tau^x\leq\delta].
\]
We have
\[
  \bar X_{\bar\tau^x}^x - x
    = \int_0^{\bar\tau^x}\bar b(\bar X_s^x)\,ds
      + \int_0^{\bar\tau^x}\bar\sigma(\bar X_s^x)\,dB_s
    \qedef J_1 + J_2.
\]
On $\{\bar\tau^x\leq\delta\}$ and for $x\in B_{3r}(x_0)$,
we have $|\bar X_{\bar\tau^x}^x - x|\geq r$, thus
$|J_1+J_2|\geq r$, hence $|J_1|\geq\frac12r$
or $|J_2|\geq\frac12r$. Therefore, for every
$q\geq1$,
\[
  \begin{aligned}
    \Prob[\bar\tau^x\leq\delta]
      &\leq \Prob[\bar\tau^x\leq\delta,|J_1|\geq\tfrac12r]
        + \Prob[\bar\tau^x\leq\delta,|J_2|\geq\tfrac12r]\\
      &\lesssim\frac1{r^q}\|\bar b\|_{L^\infty}^q
        \E[(\bar{\tau^x}\wedge\delta)^q]
        + \frac1{r^q}\E\Bigl[\Bigl|\int_0^{\bar{\tau^x}\wedge\delta}
        \bar\sigma(\bar X_s^x)\,dB_s\Bigr|^q\Bigr]\\
      &\lesssim\frac1{r^q}
        (\|\bar b\|_{L^\infty}^q + \|\bar\sigma\|_{L^\infty}^q)
        \delta^{q/2},
  \end{aligned}
\]
and in conclusion
\begin{equation}\label{e:deca2}
  \E\bigl[f(\bar X_\delta^x)\varphi_r(\bar X_\delta^x)
      \uno_{\{\bar\tau^x\leq\delta\}}\bigr]
    \lesssim \frac{\|\varphi\|_{L^\infty}}{r^q}
      (\|\bar b\|_{L^\infty}^q + \|\bar\sigma\|_{L^\infty}^q)
      \|f\|_{L^\infty}\delta^{\frac{q}2}.
\end{equation}
Before turning to the next step, we summarise what
we have proved so far. By \eqref{e:decap1}, \eqref{e:decap2}
and \eqref{e:deca}, we see that
\[
  \E[f(X_t)\varphi_r(X_t)\uno_{I_t}]
    \lesssim \sup_{x\in B_{3r}(x_0)}\Bigl|
      \E\bigl[f(\bar X_\delta^x)\varphi_r(\bar X_\delta^x)\bigr]
      \Bigr| + \|f\|_{L^\infty}\delta^{\frac{q}2}.
\]
We use this estimate and \eqref{e:decb} in \eqref{e:dec}
to finally obtain
\begin{equation}\label{e:decf}
  \E[f(X_t)\varphi_r(X_t)]
    \lesssim \sup_{x\in B_{3r}(x_0)}\Bigl|
      \E\bigl[f(\bar X_\delta^x)\varphi_r(\bar X_\delta^x)\bigr]
      \Bigr| + \|f\|_{L^\infty}\delta^{\frac{q}2},
\end{equation}
where the constant in the inequality above depends on
$r$, $\|\bar b\|_{L^\infty}$ and $\|\bar\sigma\|_{L^\infty}$.
\subsubsection{The method}

Let $\alpha\in(0,1)$ and $m\geq1$ (to be suitably
chosen later), and $h\in\R^d$ with $|h|\leq 1$.
Consider $\phi\in\Cs^\alpha_b(\R^d)$, and
apply \eqref{e:decf} to get
\begin{equation}\label{e:decf2}
  \E[(\Delta_h^m\phi)(X_t)\varphi_r(X_t)]
    \lesssim \sup_{x\in B_{3r}(x_0)}\Bigl|
      \E\bigl[(\Delta_h^m\phi)(\bar X_\delta^x)
      \varphi_r(\bar X_\delta^x)\bigr]\Bigr|
      + \|\phi\|_{L^\infty}\delta^{\frac{q}2},
\end{equation}
We will apply our method to the new process $\bar X$.
Under our standing assumptions, $\bar X$ has globally
bounded drift and globally non--singular H\"older
diffusion coefficient, so in principle the estimate
should not be any different than what we did
in Section~\ref{s:bulk}. Indeed, we consider the extra
parameter $\epsilon\in[0,\delta]$ and an
extra process $\bar Y^x$, defined as in
Section~\ref{s:bulk}, and we do the decomposition
\[
  \begin{aligned}
    \E\bigl[(\Delta_h^m\phi)(\bar X_\delta^x)
        \varphi_r(\bar X_\delta^x)\bigr]
      &= \E\bigl[(\Delta_h^m\phi)(\bar X_\delta^x)
        \varphi_r(\bar X_\delta^x)\bigr]
        - \E\bigl[(\Delta_h^m\phi)(\bar Y_\delta^x)
        \varphi_r(\bar Y_\delta^x)\bigr] + \\
      &\quad  + \E\bigl[(\Delta_h^m\phi)(\bar Y_\delta^x)
        \varphi_r(\bar Y_\delta^x)\bigr]
  \end{aligned}
\]
The \Ae is essentially the same,
\[
  \begin{aligned}
    \lefteqn{\E\bigl[(\Delta_h^m\phi)(\bar X_\delta^x)
        \varphi_r(\bar X_\delta^x)\bigr]
        - \E\bigl[(\Delta_h^m\phi)(\bar Y_\delta^x)
        \varphi_r(\bar Y_\delta^x)\bigr] =}\quad \\
      &= \E\bigl[(\Delta_h^m\phi)(\bar X_\delta^x)
        \bigl(\varphi_r(\bar X_\delta^x)
        - \varphi_r(\bar Y_\delta^x)\bigr)\bigr]
        + \E\bigl[\bigl((\Delta_h^m\phi)(\bar X_\delta^x)
        - (\Delta_h^m\phi)(\bar Y_\delta^x)\bigr)
        \varphi_r(\bar Y_\delta^x)\bigr]\\
      &\lesssim\frac{\|D\varphi\|_{L^\infty}}{r}\|\phi\|_{L^\infty}
        \E|\bar X_\delta^x - \bar Y_\delta^x|
        + \|\varphi\|_{L^\infty}[\phi]_{\Cs_b^\alpha}
        \E|\bar X_\delta^x - \bar Y_\delta^x|^\alpha.
  \end{aligned}
\]
As in Section~\ref{s:bulk}, we have
\[
  \E|\bar X_\delta^x - \bar Y_\delta^x|
    \lesssim \bigl(\|\bar b\|_{L^\infty} + 
      [\bar{\sigma}]_{C^\beta}(\|\bar b\|_{L^\infty}
      + \|\bar\sigma\|_{L^\infty})^\beta\bigr)
      \epsilon^{\frac12(1+\beta)},
\]
thus
\[
  \apperr
    \lesssim (1+\tfrac1r)\|\varphi\|_{C^1}\|\phi\|_{\Cs^\alpha_b}
      \epsilon^{\frac\alpha2(1+\beta)}.
\]

The \Pe is slightly more delicate.
By conditioning on the $\sigma$--field $\Fs_{\delta-\epsilon}$
of events known at time $\delta-\epsilon$, 
\[
  \begin{aligned}
    \E\bigl[(\Delta_h^m\phi)(\bar Y_\delta^x)
        \varphi_r(\bar Y_\delta^x)\bigr]
      &=\E\bigl[\E[(\Delta_h^m\phi)(\bar Y_\delta^x)
        \varphi_r(\bar Y_\delta^x)|\Fs_{\delta-\epsilon}]\bigr]\\
      &=\E\bigl[\E[(\Delta_h^m\phi)
        (u+\bar\sigma(u)\tilde B_\epsilon))
        \varphi_r(u+\bar\sigma(u)\tilde B_\epsilon)]
        \big|_{u=\bar Y^x_{\delta-\epsilon}}\bigr],
  \end{aligned}
\]
since
\[
  \bar Y^x_\delta
    = Y^x_{\delta-\epsilon} + \bar\sigma(Y^x_{\delta-\epsilon})(B_\delta-B_{\delta-\epsilon}),
\]
and $\tilde B_s=B_{\delta-\epsilon+s}-B_{\delta-\epsilon}$, $s\in[0,\epsilon]$,
is a Brownian motion. It remains to analyze the term
\[
  \E[(\Delta_h^m\phi)(u+\bar\sigma(u)\tilde B_\epsilon))
      \varphi_r(u+\bar\sigma(u)\tilde B_\epsilon)]
    = \int_{\R^d} (\Delta_h^m\phi)(u+y)
      \varphi_r(u+y)g_\epsilon(y)\,dy,
\]
where $g_\epsilon$ is the density of a centred Gaussian
random vector with covariance $\bar\sigma(u)\bar\sigma(u)^\star\epsilon$.
By a discrete integration by parts (see below in \eqref{e:ddformulae}),
\[
  \begin{aligned}
    \E[(\Delta_h^m\phi)(u+\bar\sigma(u)\tilde B_\epsilon))
        \varphi_r(u+\bar\sigma(u)\tilde B_\epsilon)]
      &= \int_{\R^d}\phi(u+y)
        \Delta_{-h}^m(\varphi_r(u+\cdot)g_\epsilon)(y)\,dy\\
      &\leq\|\phi\|_{L^\infty}
        \|\Delta_{-h}^m(\varphi_r(u+\cdot)g_\epsilon)\|_{L^1}.
  \end{aligned}
\]
The Leibniz formula for discrete derivatives \eqref{e:ddformulae}
yields
\[
  \|\Delta_{-h}^m(\varphi_r(u+\cdot)g_\epsilon)\|_{L^1}
    \leq\sum_{k=0}^m\binom{m}{k}\|\Delta_{-h}^kg_\epsilon\|_{L^1}
      \|(\Delta_{-h}^{m-k}\varphi_r(u+\cdot))
      (\cdot+kh)\|_{L^\infty}.
\]
On the one hand, since $\epsilon\leq 1$, and since
by the assumptions of the theorem there is $\bar\sigma_0>0$
such that
$\bar\sigma_0^2\leq\det(\bar\sigma(u)\bar\sigma(u)^\star)$
for all $u\in\R^d$,
\[
  \|\Delta_{-h}^kg_\epsilon\|_{L^1}
    \lesssim\Bigl(\frac{|h|}
      {\sqrt{\epsilon\det(\bar\sigma(u)\bar\sigma(u)^\star)}}\Bigr)^k
    \lesssim\frac1{\bar\sigma_0^k}\frac{|h|^k}{\epsilon^{\frac{k}2}},
\]
on the other hand,
\[
  \|(\Delta_{-h}^{m-k}\varphi_r(u+\cdot))(\cdot+kh)\|_{L^\infty}
    \lesssim\frac{\|D^{m-k}\varphi\|_{L^\infty}}{r^{m-k}}|h|^{m-k},
\]
thus,
\[
  \|\Delta_{-h}^m(\varphi_r(u+\cdot)g_\epsilon)\|_{L^1}
    \lesssim\frac{|h|^m}{\epsilon^{\frac{m}2}}\|\varphi\|_{C^m}
      \sum_{k=0}^m\binom{m}{k}\frac1{\bar\sigma_0^k}\frac1{r^{m-k}}
    = \|\varphi\|_{C^m}\bigl(\tfrac1r+\tfrac1{\bar\sigma_0}\bigr)^m
      \frac{|h|^m}{\epsilon^{\frac{m}2}},
\]
and in conclusion
\[
  \E\bigl[(\Delta_h^m\phi)(\bar Y_\delta^x)
      \varphi_r(\bar Y_\delta^x)\bigr]
    \lesssim \|\phi\|_{L^\infty}\|\varphi\|_{C^m}
      \bigl(\tfrac1r+\tfrac1{\bar\sigma_0}\bigr)^m
      \frac{|h|^m}{\epsilon^{\frac{m}2}}.
\]

In conclusion, by \eqref{e:decf2} and the above computations
we have that
\[
  \E\bigl[(\Delta_h^m\phi)(X_t)
      \varphi_r(X_t)\bigr]
    \lesssim\|\phi\|_{\Cs^\alpha_b}\Bigl(
      \delta^{\frac{q}2} + \epsilon^{\frac12\alpha(1+\beta)}
      + \frac{|h|^m}{\epsilon^{\frac{m}2}}\Bigr),
\]
with $q$, $\delta$, $\epsilon$ yet to be chosen,
with $0<\epsilon\leq\delta\leq 1\wedge t$. Choose
$\delta$ to be a multiple of $\epsilon$, and
$q=\alpha(1+\beta)$, so that
\[
  \E\bigl[(\Delta_h^m\phi)(X_t)
      \varphi_r(X_t)\bigr]
    \lesssim\|\phi\|_{\Cs^\alpha_b}\Bigl(
      \epsilon^{\frac\alpha2(1+\beta)}
      + \frac{|h|^m}{\epsilon^{\frac{m}2}}\Bigr),
\]
The same computations of Proposition~\ref{p:bulk}
yield the result.

We finally mention two formulae on discrete derivatives we
have used in the computations above,
\begin{equation}\label{e:ddformulae}
  \begin{gathered}
    \int_{\R^d}(\Delta_h^m\phi)(x)\psi(x)\,dx
      = \int_{\R^d}\phi(x)(\Delta_{-h}^m\psi)(x)\,dx,\\
    \Delta_h^m(\phi\psi)(x)
      = \sum_{k=0}^m \binom{m}{k}(\Delta_h^m\phi)(x)
        (\Delta_h^{m-k}\psi)(x+kh).
  \end{gathered}
\end{equation}
\section{Rougher coefficients}\label{s:rougher}

In this section we wish to discuss an extension of
the basic result on existence given in Section~\ref{s:bulk}.
We will only be able to lower the required regularity
of the drift coefficient.

We wish first to give a few remarks about the possible
limitations in the case of rough coefficients.
With Proposition~\ref{p:bulk} at hand, it is reasonable
to expect to find a threshold of regularity for the
coefficients under which the method fails. Indeed,
the key requirement is $a_0>0$ in formula~\eqref{e:bulkstart}.
In the simple case of equation~\eqref{e:roughdrifteq}
below, this would correspond to the estimate
$\apperr\lesssim\epsilon^{\alpha/2}$, that is
\[
  \E\Bigl|\int_{t-\epsilon}^t b(X_s)\,ds\Bigr|^\alpha
    \lesssim \epsilon^{\frac\alpha2}.
\] 
Notice that this estimate is readily available if
one observes that
\[
  \int_{t-\epsilon}^t b(X_s)\,ds
    = X_t - X_{t-\epsilon} - (B_t - B_{t-\epsilon}).
\]
In other words, in the case $a_0=0$, at short
times, the size of the drift is the same
as the size of the noise and in principle
one can believe that the effect of the
noise fails to be effective for densities.

It would be interesting to understand if
there is a counterexample to existence of
densities in the case $a_0\leq0$, or
simply the method fails. 
\subsection{Rougher drift coefficient}\label{s:rougherdrift}

In this section we shall study the existence
of a density for solutions of the equation
\begin{equation}\label{e:roughdrifteq}
  dX_t
    = b(t,X_t)\,dt + \,dB_t,
\end{equation}
with non--regular coefficient $b$.
We have taken a simple diffusion
coefficient to focus on the
difficulties originated by the drift.

Theorem~\ref{t:roughdrift} below
shows existence and regularity of
a density under the assumption
$b\in L^q(0,T;L^p(\R^d))$,
with $\frac{2}{q}+\frac{d}{p}<1$.
It is interesting to notice that
this same condition ensures
existence and uniqueness of
a strong solution \cite{KryRoc2005}
of \eqref{e:roughdrifteq}.
See Remark~\ref{r:kolmogorov}
for a wider discussion along
these lines.

This setting could provide a testbed
for the problem of existence of densities
in the case $a_0\leq0$ (addressed at the
beginning of Section~\ref{s:rougher}).
\begin{theorem}\label{t:roughdrift}
  Assume $b\in L^q(0,T;L^p(\R^d))$, with
  $\frac{2}{q}+\frac{d}{p}<1$. Then for
  every initial condition $x\in\R^d$
  the solution of \eqref{e:roughdrifteq}
  with initial condition $x$ has a density
  $p_x(t)$ for every $t>0$. Moreover, for
  every $\gamma<1-\frac{2}{q}$,
  \[
    \sup_{t\in(0,T]}\sup_{x\in\R^d}
      (1\wedge t)^{e_\gamma}\|p_x(t)\|_{B^\gamma_{1,\infty}}
      <\infty,
  \]
  where $e_\gamma$ is any number such that
  $e_\gamma>\frac{1-\frac1q}{1-\frac2q-\frac{d}p}\gamma$.
\end{theorem}
\begin{proof}
  Set $p'=\frac{p}{p-1}$ and $a=\frac{d}{p}$, so that
  $p'=\frac{d}{d-a}$ and thus $B_{1,\infty}^a\subset L^{p'}$.
  Set for every $\gamma>0$,
  \[
    \|p_\cdot\|_{\star,\gamma}
      \eqdef\sup_{t\in(0,T]}\sup_{x\in\R^d}
        (1\wedge t)^{e_\gamma}\|p_x(t)\|_{B^\gamma_{1,\infty}},
  \]
  where $\gamma\mapsto e_\gamma$ will be identified
  below in the proof. Now, by the H\"older inequality,
  and Sobolev's embeddings,
  \[
    \begin{aligned}
      \E\int_{t-\epsilon}^t|b(s,X_s)|\,ds
        &= \int_{t-\epsilon}^t\int_{\R^d}|b(s,y)|p_x(s,y)\,dy\\
        &\leq \int_{t-\epsilon}^t \|b(s)\|_{L^p}
          \|p_x(s)\|_{L^{p'}}\,ds\\
        &\lesssim \int_{t-\epsilon}^t\|b(s)\|_{L^p}
          \|p_x(s)\|_{B^a_{1,\infty}}\,ds\\
        &\leq \|p_\cdot\|_{\star,a}\int_{t-\epsilon}^t
          \|b(s)\|_{L^p}(1\wedge s)^{-e_a}\,ds\\
        &\leq \|p_\cdot\|_{\star,a}\|b\|_{L^q(0,T;L^p)}
          \Bigl(\int_{t-\epsilon}^t(1\wedge s)^{-e_aq'}
          \,ds\Bigr)^{\frac1{q'}}\\
        &\lesssim (1\wedge t)^{-e_a}\epsilon^{\frac1{q'}}
          \|p_\cdot\|_{\star,a}\|b\|_{L^q(0,T;L^p)},
    \end{aligned}
  \]
  where $q'=\frac{q}{q-1}$.
  We can thus estimate the \Ae with,
  \[
    \begin{multlined}[.9\linewidth]
      \apperr
        \leq [\phi]_{\Cs^\alpha_b}
          \E\Bigl[\Bigl|\int_{t-\epsilon}^t b(s,X_s)\,ds
          \Bigr|^\alpha\Bigr]\leq\\
        \leq [\phi]_{\Cs^\alpha_b}
          \E\Bigl[\int_{t-\epsilon}^t |b(s,X_s)|\,ds\Bigr]^\alpha
        \lesssim (1\wedge t)^{-\alpha e_a}\epsilon^{\alpha/q'}
          \|p_\cdot\|_{\star,a}^\alpha[\phi]_{\Cs^\alpha_b}.
    \end{multlined}
  \]
  The \Pe is as in the standard
  case (Section~\ref{s:bulk}), so,if we set
  $K_0=(1\wedge t)^{-2\alpha e_a}(1\vee\|p_\cdot\|_{\star,a}^{2\alpha})$
  and $a_0=\frac2{q'}-1$, then we have \eqref{e:bulkstart},
  therefore \eqref{e:bulkend} holds, that is
  \begin{equation}\label{e:brough}
    \begin{aligned}
      \|p_x(t)\|_{B^\gamma_{1,\infty}}
        &\lesssim K_0^{\frac\gamma{2a_0}+\delta}
          (1\wedge t)^{-\frac{1+a_0}{2a_0}\gamma-\delta}\\
        &= (1\wedge t)^{-(\frac\gamma{a_0}+2\delta)e_a
          -\frac{1+a_0}{2a_0}\gamma-\delta}
          (1\vee\|p_\cdot\|_{\star,a}^{\frac\gamma{a_0}+2\delta}).
    \end{aligned}
  \end{equation}
  This formula shows an a--priori estimate for
  $\|p_\cdot\|_{\star,\gamma}$, once we have defined,
  through this same formula, the value of
  $e_\gamma$ as
  \[
    e_\gamma
      > \frac\gamma{a_0}e_a + \frac{1+a_0}{2a_0}\gamma.
  \]
  Unfortunately the estimate, as well as the value $e_\gamma$,
  depend on $\|p_\cdot\|_{\star,a}$ and $e_a$.
  
  Notice that our assumption ensures that $a<a_0$, therefore
  if we choose $\gamma=a$, the restriction on $e_a$ reads
  \[
    e_a
      >\frac{1+a_0}{2(a_0-a)}a.
  \]
  The constant in the inequality \eqref{e:brough} does not depend
  on the initial condition, thus
  \[
    \|p_\cdot\|_{\star,a}
      \lesssim (1\vee\|p_\cdot\|_{\star,a}^{\frac{a}{a_0}+2\delta}).
  \]
  If we choose $\delta$ small enough that $\frac{a}{a_0}+2\delta<1$,
  this provides an a--priori estimate for $\|p_\cdot\|_{\star,a}$,
  with $e_a$ chosen as above.
  
  With the estimate of $\|p_\cdot\|_{\star,a}$ at hand, we look back
  at \eqref{e:brough} and see that if
  \[
    e_\gamma
      > \frac{1+a_0}{2(a_0-a)}\gamma,
  \]
  then \eqref{e:brough} provides an estimate for
  $\|p_\cdot\|_{\star,\gamma}$.
\end{proof}
\begin{remark}\label{r:kolmogorov}
  An alternative proof via a backward Kolmogorov equation
  is available to prove Theorem~\ref{t:roughdrift}.
  The idea we wish to use has been extensively used
  in results on regularization by noise, see for
  instance \cite{FlaGubPri2010}, or
  \cite{KryRoc2005}, and \cite{DubRev2015p}
  for an extension to random drifts. The idea,
  here, is to re--write the equation
  \eqref{e:roughdrifteq} as an equation with
  more regular coefficients. Notice though
  that in this way, to ``solve'' a Fokker--Planck
  equation, we end up using its dual Kolmogorov
  equation. Our original proof has the advantage
  to be based on elementary arguments that do
  not require to solve PDEs, and therefore can
  be readily used in more general cases, see
  for instance Section~\ref{s:lrougher}.
  
  Indeed, for $b\in L^q(0,T;L^p(\R^d))$,
  $p\geq2$ and $\frac{d}{p}+\frac2q<1$,
  consider the following backward parabolic equation,
  \begin{equation}
    \begin{cases}
      \partial_s U + \frac12\Delta U
          - \lambda U + b\cdot\nabla U
        = b,\qquad &s\leq t,\\
      U(t,x) = 0, &x\in\R^d.
    \end{cases}
  \end{equation}
  with
  $U\in L^q(0,t;W^{2,p}(\R^d))\cap W^{1,q}(0,t;L^p(\R^d))$
  (see \cite[Theorem 10.3]{KryRoc2005}). It is possible
  to give a quantitative estimate of the dependence on
  $\lambda$ (with a minor modification from
  \cite[Lemma 3.4]{FedFla2013}, to make the dependence
  on $\lambda$ more explicit),
  \[
    \|U\|_{L^\infty([0,t]\times\R^d}
      \leq c\lambda^{-a-\frac12}.
        \qquad
    \|\nabla U\|_{L^\infty([0,t]\times\R^d}
      \leq c\lambda^{-a}.
  \]
  for every positive $a<\frac12(1-\frac2q-\frac{d}{p})$,
  and for $\lambda$ large enough. By It\=o's formula,
  \[
    \begin{multlined}[.8\linewidth]
      \int_{t-\epsilon}^t b(s,X_s)\,ds
        = U(t,X_t) - U(t-\epsilon,X_{t-\epsilon})\\
          - \lambda\int_{t-\epsilon}^t U(s,X_s)\,ds
          - \int_{t-\epsilon}^t\nabla U(s,X_s)\cdot dB_s,
    \end{multlined}
  \]
  and this immediately allows to estimate the
  \Ae (with the optimal choice $\lambda=\epsilon^{-1}$).
\end{remark}
\subsubsection{An almost sure regularity result}

We discuss a simple application of the previous
result. The regularity of the density
we have found in Theorem~\ref{t:roughdrift}
can be used to obtain
regularity of functionals of solutions of
\eqref{e:roughdrifteq}.
Indeed, under the assumptions of Theorem~\ref{t:roughdrift},
if $f\in L^q(0,T;L^p(\R^d))$, then the function
$t\mapsto\E[\int_0^t f(r,X_r^x)\,dr]$
is H\"older continuous in $(0,T]$.
To see this, we simply use the H\"older inequality,
embeddings of Besov spaces, 
and the estimate of the previous theorem to obtain,
\[
  \begin{multlined}[.95\linewidth]
    \E\Bigl[\int_s^t f(r,X_r^x)\,dr\Bigr]
      \leq\int _s^t \|f(r)\|_{L^p}\|p_x(r)\|_{L^{p'}}\,ds\\
      \lesssim\int_s^t |f(r)\|_{L^p}
        \|p_x(r)\|_{B^a_{1,\infty}}\,ds
      \lesssim\int_s^t \|f(r)\|_{L^p}(1\wedge r)^{-e_a}\,dr\\
      \lesssim \|f\|_{L^q(L^p)}
        \Bigl(\int_s^t(1\wedge r)^{-q'e_a}\,dr\Bigr)^{\frac1{q'}}
      \lesssim\frac{(t-s)^{\frac1{q'}}}{(1\wedge s)^{e_a}}
        \|f\|_{L^q(L^p)},
  \end{multlined}
\]
where $a=\frac{d}{p}$, and $p'$ and $q'$ are the conjugate
H\"older exponents of $p$, $q$. Similar computations
show that the same function is H\"older
continuous on $[0,T]$ 
if $q'e_a<1$, that is if $\frac{2d}{p}+\frac2q<1$.
More generally, we have the following result.
\begin{proposition}
  Assume $b\in L^q(0,T;L^p(\R^d))$, with
  $\frac{2}{q}+\frac{2d}{p}<1$. Let $X^x$
  be the solution of \eqref{e:roughdrifteq}
  with initial condition $x\in\R^d$, and
  let $p_x(t)$ be its density for every $t>0$.
  If $f\in L^q(0,T;L^p(\R^d))$, then the map
  \[
    t\mapsto\int_0^t f(r,X_r^x)\,dr
  \]
  is {a.~s.} H\"older continuous on $[0,T]$
  with exponent smaller than $\frac1{q'}-e_a$, where
  $q$ is the conjugate H\"older exponent of $q$
  and $a=\frac{d}{p}$.
\end{proposition}
\begin{proof}
  The proof is elementary and uses the Markov property
  and Kolmogorov's continuity theorem. We only give a
  sketch under the assumption that $f$ does not depend
  from $t$. The general case is entirely similar.
  
  Let $a=\frac{d}{p}$, and set for brevity
  $g(r)=(1\wedge r)^{-e_a}$.
  By the Markov property, if $r_1<r_2<\dots<r_k$,
  \[
    \begin{aligned}
      \E[f(X_{r_1}^x)f(X_{r_2}^x)\dots f(X_{r_k}^x)]
        &= \E\bigl[f(X_{r_1}^x)f(X_{r_2}^x)\dots f(X_{r_{k-1}}^x)
          \E[f(X_{r_k}^x)|\Fs_{r_{k-1}}]\bigr]\\
        &= \E\bigl[f(X_{r_1}^x)f(X_{r_2}^x)\dots f(X_{r_{k-1}}^x)
          \E[f(X_{r_k-r_{k-1}}^y)]_{y=X^x_{r_{k-1}}}\bigr]\\
        &\lesssim \|f\|_{L^p}g(r_k-r_{k-1})
          \E\bigl[f(X_{r_1}^x)f(X_{r_2}^x)\dots f(X_{r_{k-1}}^x)]\\
        &\lesssim\dots
        \lesssim \|f\|_{L^p}^k g(r_1)
         g(r_2-r_1)\dots
         g(r_k-r_{k-1}), 
    \end{aligned}
  \]
  since, as above,
  \[
    \E[f(X_r^x)]
      \leq \|f\|_{L^p}\|p_x(r)\|_{L^{p'}}
      \leq \|f\|_{L^p}\|p_x(r)\|_{B^a_{1,\infty}}
      \lesssim \|f\|_{L^p}g(r).
  \]
  Thus, by $k$ changes of variables,
  \[
    \begin{aligned}
      \lefteqn{\E\Bigl[\Bigl(\int_s^t
          f(X_r^x)\,dr\Bigr)^k\Bigr]=}\quad\\
        &= \int_s^t\ldots\int_s^t \E[f(X_{r_1}^x)f(X_{r_2}^x)
          \dots f(X_{r_k}^x)]\,dr_k\dots\,dr_2\,dr_1\\
        &= k!\int_s^t\,dr_1\int_{r_1}^t\,dr_2
          \dots\int_{r_{k-1}}^t\,dr_k
          \E[f(X_{r_1}^x)f(X_{r_2}^x)\dots f(X_{r_k}^x)]\\
        &= k!\|f\|_{L^p}^k
          \smashoperator{\int_s^t}\,dr_1
          \smashoperator{\int_{r_1}^t}\,dr_2
          \dots\smashoperator{\int_{r_{k-1}}^t}\,dr_k
          g(r_1)g(r_2-r_1)\dots g(r_k-r_{k-1})\\
        &= k!\|f\|_{L^p}^k
          \smashoperator{\int_0^{t-s}}g(t-r_1)
          \smashoperator{\int_0^{r_1}}g(r_1-r_2)
          \smashoperator{\int_0^{r_2}}g(r_2-r_3)\dots
          \smashoperator{\int_0^{r_{k-1}}}g(r_{k-1}-r_k)
          \,dr_k\dots\,dr_2\,dr_1.
    \end{aligned}
  \]
  If $t-s\leq 1$, then
  $r_k\leq r_{k-1}\leq\dots\leq r_1\leq t-s\leq1$
  and
  \[
    \int_0^{r_{k-1}}g(r_{k-1}-r_k)\,dr_k
      = \frac{r_{k-1}^{1-e_a}}{1-e_a}
      \lesssim (t-s)^{1-e_a},
  \]
  therefore
  \[
    \E\Bigl[\Bigl(\int_s^t f(X_r^x)\,dr\Bigr)^k\Bigr]
      \lesssim (t-s)^{(k-1)(1-e_a)}
        \int_0^{t-s}g(t-r_1)\,dr_1
      \lesssim (t-s)^{k(1-e_a)},
  \]
  since
  \[
    \int_0^{t-s}g(t-r_1)\,dr_1
      \lesssim \frac{t-s}{(1\wedge t)^{e_a}}
      \leq (t-s)^{1-e_a}.\qedhere
  \]
\end{proof}
\begin{remark}
  We remark that the above proposition
  \emph{\bfseries does not} show that the stochastic
  flow generated by \eqref{e:roughdrifteq}
  is H\"older continuous.
\end{remark}
\subsection{Rougher diffusion coefficient}
  \label{s:rougherdiffusion}

We wish to briefly discuss some difficulties related
to the extension of the ``core'' method with
rougher diffusion coefficients.

First of all, regularity of the diffusion coefficient
is a more delicate issue, and regularity itself might
not be as significant as for the drift coefficient
in view of densities. For instance by
the Levy characterization theorem we know 
if $\sigma:\R^d\to\R^{d\times d'}$
satisfies $|\sigma(x)|=1$ for all $x$, then the
solution of
\begin{equation}\label{e:sdern}
  dX_t
    = \sigma(X_t)\,dB_t,
\end{equation}
is a Brownian motion and thus has a smooth density.

On the other hand the method we have introduced
seems to strongly depend on an evaluation of
the increments of $\sigma$. It would be thus reasonable
to expect results if we lower the summability of
the ``derivatives'' of $\sigma$, namely by
requiring that $\sigma\in B^\beta_{p,q}$ for some
$p,q\geq1$ but finite (recall that $\Cs^\beta_b
=B^\beta_{\infty,\infty}$).

A strategy to estimate the \Ae
could be as in Theorem~\ref{t:roughdrift}. Assume
for instance $\sigma$ is non--singular and
$\sigma\in L^\infty\cap B^\beta_{p,q}$,
and let $p_x(s)$ be the density of the solution $X_s^x$
of \eqref{e:sdern} with initial condition
$X_0=x$, fix $t>0$ and let $Y$ be the auxiliary process
introduced in Section~\ref{s:bulk}.
The \Ae would be
\[
  \apperr
    \lesssim [\phi]_{\Cs^\alpha_b}\Bigl(\int_{t-\epsilon}^t
      \E[|\sigma(X_r^x)-\sigma(X_{t-\epsilon}^x)|^2]\,dr
      \Bigr)^{\frac\alpha2},
\]
and the Markov property yields
\[
  \E[|\sigma(X_r^x)-\sigma(X_{t-\epsilon}^x)|^2]
    = \int_{\R^d}\int_{\R^d}\E[|\sigma(z)-\sigma(y)|^2]
      p_y(t-s,z)p_x(s,y)\,dz\,dy.
\]
On the one hand the above formula contains increments
of $\sigma$, that could be estimate using the regularity
of $\sigma$. On the other hand the term $p_y(t-s,\cdot)$
gives a too singular contribution, if we expect a
singularity in time as in Theorem~\ref{t:roughdrift}.
In other words, the estimate of the term above
is successful only when $\beta>\frac{d}{p}$,
that is whenever $\sigma$, by Sobolev's embeddings,
is a H\"older function.
\section{More regularity - II}\label{s:morereg2}

In this section we shall see that the method introduced
in Section~\ref{s:bulk} is not optimal, and will suggest
a partial probabilistic proof that goes in the direction
of the optimal result.

Consider for simplicity the problem with constant diffusion
(although our considerations equally hold with
a non singular--diffusion) and bounded
drift,
\[
  dX_t
    = b(X_t)\,dt + dB_t,
\]
with initial condition $X_0=x$. The computations in 
Section~\ref{s:bulk} show that $X_t$ has a density
$p_x(t)\in B^{1-}_{1,\infty}$. It is easy to be
convinced though that the expected regularity
is $B^{1-}_{\infty,\infty}$ (that is H\"older).
Indeed, ideally the density should solve the
associated Fokker--Planck equation
\[
  \partial_t p_x
    = \frac12\Delta p_x - \nabla\cdot(b p_x),
\]
with initial condition $p_x(0)=\delta_x$, or
\[
  p_x(t,y)
    = g(t,y-x)
      + \int_0^t \nabla g(t-s)\star(b p_x)(y)\,dt,
\]
where $g$ is the heat kernel. It is not difficult,
using this mild formulation, to prove that
\[
  \begin{aligned}
    \|\Delta_h^2 p_x(t)\|_{L^1}
      &\leq \|\Delta_h^2 g(t)\|_{L^1}
         + \Bigl\|\int_0^t (\Delta_h^2\nabla g(t-s))
         \star(bp_x)\,dt\Bigr\|_{L^1}\\
      &\leq \|\Delta_h^2 g(t)\|_{L^1}
        + \|b\|_{L^\infty}\int_0^t
        \|\Delta_h^2\nabla g(t-s)\|_{L^1}\,ds,
  \end{aligned}
\]
and conclude with estimates for the heat kernel.
Similar computations (with a caveat) also show
H\"older bounds.
The trick is to bound the quantity
\[
  \Fc_T
    = \sup_{[0,T]} t^{\frac12d}\|f(t)\|_{L^\infty},
\]
although we need to know \emph{a--priori} that this quantity
is finite. A similar quantity is considered in
Section~\ref{s:rougher}.

We turn to a result that improves slightly
(but without getting the optimal result)
the summability of the density. We will work
under the same assumptions of Section~\ref{s:bulk}.
\begin{theorem}
  Let $b\in L^\infty(\R^d)$ and
  $\sigma\in C^\beta_b(\R^d;\R^{d\times d'})$,
  with $\beta\in(0,1)$. Assume moreover
  \eqref{e:bulknonsingular}.
  Let $X^x$ be a solution
  of \eqref{e:toy} with initial condition
  $x\in\R^d$. Then the density $p_x(t)$ of
  $X_t^x$ is in $B^a_{p,\infty}(\R^d)$ for
  every $a<\beta$ and every $p<\frac{d}{d-\beta}$.
\end{theorem}
\begin{proof}
  By the computations in Section~\ref{s:bulk}
  we know that under the standing assumptions
  the density $p_x(t)\in B^a_{1,\infty}$ for
  every $a<\beta$. By Sobolev's embeddings,
  $p_x(t)\in L^p$ for every $p<\frac{d}{d-\beta}$.
  
  Fix $p\in(1,\frac{d}{d-\beta})$ and denote by
  $q$ the conjugate H\"older exponent of $p$.
  Notice that, by our choice of $p$, we have
  that $q>\frac{d}\beta$.
  
  We will use the smoothing Lemma~\ref{l:smoothing2}.
  To this end let $\phi\in F^\alpha_{q,\infty}(\R^d)$,
  with $\alpha\in(\frac{d}{q},1)$. With these
  values of $\alpha$ and $q$, we can use the
  characterization of $F^\alpha_{q,\infty}$ in terms
  of differences given in formula~\eqref{e:fseminorm}. 
  In particular, if we set
  \begin{equation}\label{e:lizorkin}
    \Phi(x)
      \eqdef\sup_{h\neq0}\frac{|\Delta_h\phi(x)|}{|h|^\alpha},
  \end{equation}
  then $\Phi\in L^q(\R^d)$,
  $[\phi]_{B^\alpha_{q,\infty}}=\|\Phi\|_{L^q}$,
  and $|\Delta_h\phi(x)|\leq |h|^\alpha\Phi(x)$.
  As usual, we consider $\E[\Delta^m_h\varphi(X_t)]$
  and split it into \Pe
  and \Ae.
  
  The \Pe is not the
  source of issues. We only have to make a H\"older
  inequality in formula~\eqref{e:bulkpe} in
  $L^p-L^q$ rather than $L^1-L^\infty$,
  to get
  \[
    \prober
      = \E[\Delta_h^m\phi(Y^\epsilon_t)]
      \lesssim\epsilon^{-\frac{d}{2q}}
        \frac{|h|^m}{\epsilon^{\frac{m}2}}\|\phi\|_{L^q},
  \]
  where $Y^\epsilon$ is the process defined
  in formula~\ref{e:auxiliary}.

  We turn to the \Ae,
  \[
    \apperr
      = \E[\phi(X_t^x) - \phi(Y_t^\epsilon)]
      = \E\bigl[\E[\phi(X_t^x) - \phi(Y_t^\epsilon)
        |X_{t-\epsilon}]\bigr].
  \]
  Notice that $Y^\epsilon_t$ is a function of
  $X_{t-\epsilon}^x$ and of a Brownian motion
  $\tilde B_s=B_{t-\epsilon+s}-B_{t-\epsilon}$,
  $s\geq0$ that is independent from the history
  until time $t-\epsilon$, and $X$ is a Markov
  process. Thus
  \[
    \E[\phi(X_t^x) - \phi(Y_t^\epsilon)|X_{t-\epsilon}^x]
      = \E[\phi(X_\epsilon^y)
        - \phi(y+\sigma(y)\tilde B_\epsilon)]
        |_{y=X^x_{t-\epsilon}},
  \]
  and
  \[
    X^y_\epsilon
      = y + \sigma(y)\tilde B_\epsilon
        + \int_0^\epsilon b(X^y_s)\,ds
        + \int_0^\epsilon(\sigma(X^y_s)-\sigma(y))\,d\tilde B_s.
  \]
  Therefore, by the H\"older inequality, and with
  the same computations used to obtain the estimate
  \eqref{e:bulkae}
  \[
    \begin{aligned}
      \lefteqn{\E[|\phi(X_\epsilon^y)
          - \phi(y+\sigma(y)\tilde B_\epsilon)]\leq}\qquad\\
        &\leq \E\Bigl[\Bigl|\int_0^\epsilon b(X^y_s)\,ds
          + \int_0^\epsilon(\sigma(X^y_s)-\sigma(y))
          \,d\tilde B_s\Bigr|^\alpha
          \Phi(y + \sigma(y)\tilde B_\epsilon)\Bigr]\\
        &\leq\E\Bigl[\Bigl|\int_0^\epsilon b(X^y_s)\,ds
          + \int_0^\epsilon(\sigma(X^y_s)-\sigma(y))
          \,d\tilde B_s\Bigr|^{\alpha p}\Bigr]^{\frac1p}
          \E[|\Phi(y + \sigma(y)\tilde B_\epsilon)|^q]^{\frac1q}\\
        &\lesssim \epsilon^{\frac\alpha2(1+\beta)}
          \E[|\Phi(y + \sigma(y)\tilde B_\epsilon)|^q]^{\frac1q}.
    \end{aligned}
  \]
  We thus have, using the computations above
  and again the H\"older inequality,
  \[
    \begin{aligned}
      \apperr
        &= \E\bigl[\E[\phi(X_\epsilon^y)
          - \phi(y+\sigma(y)\tilde B_\epsilon)]
          |_{y=X^x_{t-\epsilon}}\bigr]\\
        &= \int_{\R^d}\E[\phi(X_\epsilon^y)
          -\phi(y+\sigma(y)\tilde B_\epsilon)]
          p_x(t-\epsilon,y)\,dy\\
        &\lesssim \epsilon^{\frac\alpha2(1+\beta)}\int_{\R^d}
          \E[|\Phi(y + \sigma(y)\tilde B_\epsilon)|^q]^{\frac1q}
          p_x(t-\epsilon,y)\,dy\\
        &\lesssim \epsilon^{\frac\alpha2(1+\beta)}
          \|p_x(t-\epsilon)\|_{L^p}\Bigl(\E\int_{\R^d}
          |\Phi(y + \sigma(y)\tilde B_\epsilon)|^q
          \,dy\Bigr)^{\frac1q}.
    \end{aligned}
  \]
  Let $g_{\sigma(y),\epsilon}$ be the density
  of $y + \sigma(y)\tilde B_\epsilon$, then
  by the non--degeneracy assumption on $\sigma$
  it is easy to see that
  \[
    g_{\sigma(y),\epsilon}(z)
      \lesssim \frac1{(2\pi\epsilon)^{d/2}}
        \e^{-\frac{c}{\epsilon}|z-y|^2},
  \]
  for a constant $c$ that depends on $\sigma$,
  thus
  \[
    \E\int_{\R^d}|\Phi(y + \sigma(y)\tilde B_\epsilon)|^q\,dy
      = \int_{\R^d}|\Phi(z)|^q\int_{\R^d}
        g_{\sigma(y),\epsilon}(z)\,dy\,dz
      \lesssim \|\Phi\|_{L^q}^q
      \sim [\phi]_{F^\alpha_{q,\infty}}.
  \]
  In conclusion,
  \[
    \apperr
      \lesssim \epsilon^{\frac\alpha2(1+\beta)}
        \|p_x(t-\epsilon)\|_{L^p}[\phi]_{F^\alpha_{q,\infty}}.
  \]
  Since we work with $\epsilon<1\wedge\frac{t}2$,
  by \eqref{e:bulkend} we have that
  \[
    \|p_x(t-\epsilon)\|_{L^p}
      \lesssim\|p_x(t-\epsilon)\|_{B^{d/q}_{1,\infty}}
      \lesssim c_t
  \]
  Using the same computations of Proposition~\ref{p:bulk}
  and the smoothing Lemma~\ref{l:smoothing2}, we finally
  obtain that $p_x(t)\in B^a_{p,\infty}$ for every
  $a<\beta$.
\end{proof}
\begin{remark}
  With a more careful analysis of the proof of
  Proposition~\ref{p:bulk}, it is not difficult
  to obtain an estimate of the norm of the
  density in $B^a_{p,\infty}$ in terms of $t$,
  as in \eqref{e:bulkend}.
\end{remark}
\begin{remark}
  Notice that, since by the theorem above we now
  know that the density is in
  $B^{\beta-}_{\frac{d}{d-\beta}-,\infty}$,
  we can deduce a stronger summability (this was
  the starting point of the proof above).
  Unfortunately the limitation in the characterization
  \eqref{e:fseminorm} of Triebel--Lizorkin spaces
  prevents us to iterate the above proof and deduce
  H\"older bounds. 
\end{remark}
\begin{remark}\label{r:holder}
  In the case $d=1,\beta>\frac12$ we have the
  embedding of $B^{\beta-}_{\frac{d}{d-\beta}-,\infty}$
  in spaces of H\"older functions. We can thus
  conclude that the density is H\"older
  continuous, as expected.
\end{remark}
\section{Examples and applications}\label{s:apps}

In this section we show a few applications of the
simple methods and its improvements illustrated in
the first part of the paper.

Before presenting the examples,
we wish to notice that,
in the estimate of the
\Pe, there are two main key points. The
first is the estimate of the small time
asymptotics of the ``noise'' part of the
equation. The second is that the method
we have illustrated depends very much
on the fact that for times close to the
final time the noise is independent from
the past.
In principle this might rule out,
for instance, processes such as the
fractional Brownian motion as the
source of noise. In the particular
case of the fractional Brownian
motion though, one can use a
suitable integral representation,
and in that case is easy to be convinced
that most of the results (and in particular
the basic results of Section~\ref{s:bulk})
hold true with minimal modifications
(essentially we have the Hurst index
as the value of the parameter $\theta$ in
Proposition~\ref{p:bulk}).

As it regards the second key point, we notice
that, in view of the analysis of infinite dimensional
problems, our method seems to be tailored to the white
noise, and it seems so far unlikely it can be applied
to problems driven by a spatial noise, with no temporal
component (see for instance \cite{CanFriGas2017} for results
in this direction based on Malliavin's calculus).
\subsection{A path-dependent SDE}

Here we use only the basic method (from
Section~\ref{s:bulk}), because as such we
do not have a Markov evolution.
This result can be essentially found
already in \cite{BalCar2014}, we
provide it for the purpose of illustrating
the method.
Notice that the example includes,
with suitable adjustments, also
the case of equations with delay.

Given $T>0$ and two integers $d,d'\geq1$, let
$b:[0,T]\times C([0,T];\R^d)\to\R^d$ and
$\sigma:[0,T]\times C([0,T];\R^d)\to\R^{d\times d'}$
be such that for every $t\geq0$, if
$\omega|_{[0,t]}=\omega'|_{[0,t]}$,
then $b(t,\omega)=b(t,\omega')$ and
$\sigma(t,\omega)=\sigma(t,\omega')$.
In other words, $b$ and $\sigma$, when
evaluated at time $t$, depend only on
the part of the path up to time $t$.
This is to ensure adaptness in the
equation.

Given $0\leq s\leq t\leq T$, set
for every $\omega\in C([0,T];\R^d)$,
\[
  \delta_{s,t}(\omega)
    \eqdef\sup_{r\in[s,t]}|\omega_r-\omega_s|.
\]
Assume that
\begin{itemize}
  \item $b,\sigma$ are bounded,
  \item there are $c>0$ and $\beta\in(0,1)$ such that
    \[
      |\sigma(t,\omega)-\sigma(s,\omega)|\leq
        c d_{s,t}(\omega)^\beta,
    \]
    for every $\omega\in C([0,T];\R^d)$ and
    every $0\leq s\leq t\leq T$,
  \item there is $\lambda_0>0$ such that
    $\sigma(t,\omega)\sigma(t,\omega)^\star\geq\lambda_0 I$,
    for every $t\in[0,T]$ and every $\omega\in C([0,T];\R^d)$.
\end{itemize}
Consider the following path-dependent stochastic equation,
\begin{equation}\label{e:pathsde}
  dX_t
    = b(t,\Xbo)\,dt + \sigma(t,\Xbo)\,dB_t,
\end{equation}
where $(B_t)_{t\geq0}$ is a $d'$-dimensional Brownian
motion, and $\Xbo=(X_t)_{t\in[0,T]}$.
\begin{proposition}
  Under the above assumption, if $(X_t)_{t\geq0}$ is
  a solution of \eqref{e:pathsde}, then for every 
  $t\in(0,T]$ the random variable $X_t$ has a
  density with respect to the Lebesgue measure.
  Moreover, the density is in $B^a_{1,\infty}(\R^d)$
  for every $a<\beta$. 
\end{proposition}
\begin{proof}
  We use the auxiliary process \eqref{e:auxiliary}. Notice
  that by our first assumption the term $\sigma(t-\epsilon,\Xbo)$
  is measurable with respect to the history up to $t-\epsilon$.
  
  We have,
  \[
    X_t - Y_t^\epsilon
      = \int_{t-\epsilon}^t b(s,\Xbo)\,ds
        + \int_{t-\epsilon}^t \bigl(\sigma(s,\Xbo)-\sigma(t-\epsilon,\Xbo)\bigr)\,dB_s,
  \]
  thus the \Ae is,
  \[
    \begin{aligned}
      \apperr
        &\lesssim [\varphi]_{\Cs^\alpha}
          \Bigl(\|b\|_{L^\infty}\epsilon^\alpha
          + \E\Bigl[\Bigl|\int_{t-\epsilon}^t
          \bigl(\sigma(s,\Xbo)-\sigma(t-\epsilon,\Xbo)\bigr)
          \,dB_s\Bigr|^\alpha\Bigr]\Bigr)\\
        &\lesssim [\varphi]_{\Cs^\alpha}
          \Bigl(\|b\|_{L^\infty}\epsilon^\alpha
          + \E\Bigl[\int_{t-\epsilon}^t
          |\sigma(s,\Xbo)-\sigma(t-\epsilon,\Xbo)|^2
          \,ds\Bigr]^{\frac\alpha2}\Bigr)\\
        &\lesssim [\varphi]_{\Cs^\alpha}
          \Bigl(\|b\|_{L^\infty}\epsilon^\alpha
          + \E\Bigl[\int_{t-\epsilon}^t
          \delta_{t-\epsilon,s}(\Xbo)^{2\beta}
          \,ds\Bigr]^{\frac\alpha2}\Bigr).
    \end{aligned}
  \]
  To estimate $\E[\delta_{t-\epsilon,s}(\Xbo)^{2\beta}]$,
  we notice that
  \[
    \delta_{t-\epsilon,s}(\Xbo)
      \leq \epsilon\|b\|_{L^\infty}
        + \sup_{[t-\epsilon,t]}\Bigl|
        \int_{t-\epsilon}^s\sigma(r,\Xbo)\,dB_r\Bigr|,
  \]
  so that $\E[\delta_{t-\epsilon,s}(\Xbo)^{2\beta}]
  \leq\epsilon^\beta$ and
  $\apperr\lesssim [\varphi]_{\Cs^\alpha}
  \epsilon^{\frac\alpha2(1+\beta)}$.
  
  By our second assumption, the \Pe
  can be estimated as in Section~\ref{s:bulk}, to get
  $\prober\lesssim \epsilon^{-m/2}|h|^m\|\varphi\|_{L^\infty}$.
  Proposition~\ref{p:bulk} concludes the proof.
\end{proof}
\begin{remark}
   Notice that in the derivation of the results of
   Sections~\ref{s:local}, \ref{s:rougher}, and \ref{s:morereg2},
   the Markov property was used at same stage, thus they cannot
   be adapted to this setting. 
\end{remark}
\begin{remark}
  The result we have obtained here is slightly less
  general than in \cite[Section 3.1]{BalCar2014}.
  Indeed, in \cite{BalCar2014} they have the same assumptions,
  but they
  have a weaker form of continuity of the diffusion
  coefficient, namely,
  \[
    |\sigma(t,\omega) - \sigma(s,\omega)|
      \lesssim (-\log\delta_{s,t}(\omega))^{-2-\epsilon},
  \]
  for some $\epsilon>0$. The density in this case
  is in $L^{e_{\log}}$, the Orlicz space with Young
  function $e_{\log}(u)=(1+|u|)\log(1+|u|)$. We believe
  that in principle, if one would consider Besov
  spaces with norm as in \cite{BalCar2014}, one
  might extend the result we have given above
  with this weaker assumption. We do not consider
  this, due to the analytical difficulties
  involved.
\end{remark}
\subsection{Lévy noise driven SDEs}

In this section we wish to slightly extend
the results of \cite{DebFou2013} in the direction
of Sections~\ref{s:rougher} and \ref{s:morereg2}.
We briefly recall the setting from \cite{DebFou2013}.
\begin{assumption}\label{a:levy}
  Let $(Z_t)_{t\geq0}$ be a Lévy process with
  characteristics function $k\mapsto\e^{-t\Psi(k)}$,
  with
  \[
    \Psi(u)
      = \int_{\R^d}(1-\e^{\im\scalar{u,y}}
        + \im\scalar{u,y}\uno_{\{|y|\leq 1\}})\,\mu(dy),
  \]
  and there is $\alpha\in(0,2)$ such that
  the Lévy measure $\mu$ satisfies,
  \begin{itemize}
    \item $\int_{|y|\geq1}|y|^\gamma\mu(dy)<\infty$
      for all $\gamma\in[0,\alpha)$,
    \item there is $c>0$ such that
      $\int_{|y|\leq\lambda}|y|^2\mu(dy)\leq c\lambda^{2-\alpha}$
      for all $\lambda\in(0,1]$,
    \item there are $c>0$, $r>0$ such that
      $\int(1-\cos\scalar{u,y})\mu(dy)\geq c|u|^\alpha$,
      for all $u\in\R^d$ with $|u|\geq r$.
  \end{itemize}
\end{assumption}
We consider the following equation, driven by the
$\alpha$-stable-like process $(Z_t)_{t\geq0}$
defined above,
\begin{equation}\label{e:levy}
  dX_t
    = b(t,X_t)\,dt + \sigma(X_t)\,dZ_t,
\end{equation}
where $b:[0,\infty)\times\R^d\to\R$ is bounded measurable,
and $\sigma\in C^\beta_b(\R^d;\R^{d\times d'})$, with
\begin{equation}\label{e:lnondegenerate}
  \sigma(y)\sigma(y)^\star>0,
    \qquad\text{uniformly in }y\in\R^d.
\end{equation}
Here the non--degeneracy assumption can be
weakened following the lines of
Section~\ref{s:singular}.

In the rest of the section we will restrict to the (simpler)
case $\alpha\in(1,2)$. While this greatly simplifies
the computations
of Section~\ref{s:lmorereg}, it is a necessary assumption in
Section~\ref{s:lrougher} to obtain sufficiently
smallness in the \Ae.
\subsubsection{More regularity}\label{s:lmorereg}

In this section we will adapt the ideas
of Section~\ref{s:morereg2} to
problem~\eqref{e:levy}. From
\cite{DebFou2013} we know that,
under the above assumptions on
the coefficients and on the driving
process,
any solution of \eqref{e:levy} has
a density in $B^s_{1,\infty}$, for
$s<\beta\wedge(\alpha-1)$.
We wish to improve the summability
index. As far as we know, under the
standing assumptions on the coefficients
there is not yet a general result of existence
or uniqueness for~\eqref{e:levy}.
So in the rest of this section we
assume that either there is a solution
of~\eqref{e:levy} that is a Markov
process, or that there is a solution
that can be obtained by approximation
of \eqref{e:levy} with smooth coefficients.
In the latter case the Besov bound
ensures the uniform integrability
of the approximating densities and
thus that the limit problem
has a density with the same
Besov regularity.
\begin{theorem}\label{t:lmorereg}
  Let $(Z_t)_{t\geq0}$ be a Lévy process as in
  Assumption~\ref{a:levy}, with $\alpha>1$,
  let $b$ be bounded measurable and
  $\sigma\in C^\beta_b(\R^d;\R^{d\times d'})$
  such that \eqref{e:lnondegenerate} holds,
  and set $\kappa=\min(\alpha-1,\beta)$.
  For every $x\in\R^d$
  let $(X^x_t)_{t\geq0}$ be a solution (as
  specified above) of \eqref{e:levy}.
  Then for every $t>0$ and $x\in\R^d$, the
  density $p_x(t)$ of the random
  variable $X_t^x$ is in
  $B^a_{p,\infty}(\R^d)$ for every
  $p\in(1,\frac{d}{d-\kappa})$ and
  every $a<\kappa(1\wedge\frac\alpha{p})$.
\end{theorem}
We proceed with the proof of the additional
regularity of the density.
We will need first a slight modification
of \cite[Lemma 3.3]{DebFou2013}
(which in turn is a quantitative version of
\cite[Theorem 1.2]{SchSztWan2012}).
\begin{lemma}\label{l:lpe}
  Let $(Z_t)_{t\geq0}$ be a Lévy process as in
  Assumption~\ref{a:levy}, and let
  $g_t$ be the density of $Z_t$, for each $t>0$. Then for
  all integers $m\geq1$,
  all $h\in\R^d$, $|h|\leq 1$, and all $p\in[1,\infty]$,
  \[
    \|\Delta_h^m g_t\|_{L^p}
      \lesssim (1\wedge t)^{-\frac{m}{\alpha}-\frac{d}{\alpha q}}|h|^m,
  \]
  where $q$ is the conjugate H\"older exponent of $p$.
\end{lemma}
\begin{proof}[Proof of Theorem~\ref{t:lmorereg}]
  Define for $s\geq t-\epsilon$,
  \begin{equation}\label{e:laux}
    Y^{y,\epsilon}_s
      = y + \sigma(y)(Z_t-Z_{t-\epsilon})
  \end{equation}
  and set $Y_s^\epsilon=Y_s^{X_{t-\epsilon},\epsilon}$.
  Up to the change of the driving process, this
  is the same as the auxiliary
  process~\eqref{e:auxiliary}.
  
  We proceed as in Section~\ref{s:morereg2}. Let
  $p_x$ be the density
  of the solution of \eqref{e:levy} with
  initial condition $x$.
  We know that $p_x(t)\in B^a_{1,\infty}$
  for all $a<\kappa$, then
  by Sobolev's embeddings $p_x\in L^p$ for all
  $p<\frac{d}{d-\kappa}$.
  Fix $p\in(1,\frac{d}{d-\kappa})$,
  and let $\phi\in F^\theta_{q,\infty}(\R^d)$,
  with $\theta\in(\frac{d}{q},1)$,
  where $q$ is the conjugate H\"older
  exponent of $p$.
  Define $\Phi$ as in \eqref{e:lizorkin}, so that
  $\Delta_h\phi(x)\leq |h|^\theta\Phi(x)$,
  $\Phi\in L^q(\R^d)$ and $\|\Phi\|_{L^q}=[\phi]_{F^\theta_{q,\infty}}$.
  
  The \Pe is obtained
  through the previous lemma and the non--degeneracy
  assumption on $\sigma$,
  \[
    \prober
      = \E[\E[\Delta_h^m\phi(Y^{y,\epsilon}_\epsilon)]_{y=X_{t-\epsilon}}]
      \lesssim \epsilon^{-\frac{m}{\alpha}-\frac{d}{\alpha q}}|h|^m\|\phi\|_{L^q}.
  \]
  We turn to the analysis of the \Ae.
  We use the Markov property,
  \[
    \apperr
      = \E[\E[\phi(X_\epsilon^y) - \phi(Y_\epsilon^{y,\epsilon})]_{y=X_{t-\epsilon}}],
  \]
  and by the H\"older inequality, for each $y\in\R^d$,
  \[
    \begin{aligned}
      \E[\phi(X_\epsilon^y) - \phi(Y_\epsilon^{y,\epsilon})]
        &\leq\E[|X_\epsilon^y-Y_\epsilon^{y,\epsilon}|^\theta
          \Phi(Y_\epsilon^{y,\epsilon})]\\
        &\leq \E[|X_\epsilon^y-Y_\epsilon^{y,\epsilon}|^{\theta p}]^{\frac1{p}}
          \E[|\Phi(Y_\epsilon^{y,\epsilon})|^q]^{\frac1{q}}\\
        &\lesssim \epsilon^{\frac\theta\alpha(\kappa+1)}
          \E[|\Phi(Y_\epsilon^{y,\epsilon})|^q]^{\frac1{q}}
    \end{aligned}
  \]
  where we have used \cite[Lemma 3.1-(ii)]{DebFou2013},
  and we need $\theta p<\alpha$ (this will limit the
  final regularity when $\alpha<p$). We thus have,
  \[
    \begin{aligned}
      \apperr
        &\lesssim \epsilon^{\frac\theta\alpha(\kappa+1)} \int_{\R^d}
          \E[|\Phi(Y_\epsilon^{y,\epsilon})|^{q_1}]^{\frac1{q_1}}
          p_x(t-\epsilon,y)\,dy\\
        &\leq \epsilon^{\frac\theta\alpha(\kappa+1)}
          \|p_x(t-\epsilon)\|_{L^p}
          \Bigl(\E\int_{\R^d}|\Phi(Y_\epsilon^{y,\epsilon})|^q\,dy\Bigr)^{\frac1q}.
    \end{aligned}
  \]
  First, we need a bound of $\|p_x(t-\epsilon)\|_{L^p}$. By our
  choice of $p$, $B^a_{1,\infty}\subset L^p$ for a suitable
  $a<\kappa$. Proposition~\ref{p:bulk}, together with the estimates
  in \cite{DebFou2013}, provides an estimate in time of
  $\|p_x(t-\epsilon)\|_{B^a_{1,\infty}}$. With our standard
  choice $\epsilon\leq\frac{t}2$, we readily have that
  $\|p_x(t-\epsilon)\|_{L^p}$ is bounded by a number
  that depends on $t$, but that is uniform in $x$ and $\epsilon$.
  
  It remains to estimate the last term. Denote by $g_{y,\epsilon}$
  the density of $\sigma(y)(Z_t-Z_{t-\epsilon})$, and by $g_\epsilon$
  the density of $Z_t-Z_{t-\epsilon}$. It is easy to see that
  $g_{y,\epsilon}(z)=\det(\sigma(y))^{-1}g_\epsilon(\sigma(y)^{-1}z)$,
  thus
  \[
    \E\int_{\R^d}|\Phi(Y_\epsilon^{y,\epsilon})|^q\,dy
      = \int_{\R^d}|\Phi(z)|^q\Bigl(\int_{\R^d}
        \det(\sigma(y))^{-1}g_\epsilon(\sigma(y)^{-1}(z-y))
        \,dy\Bigr)\,dz,
  \]
  and it is sufficient to prove that the inner
  integral is uniformly bounded in $z$. This
  follows from computations similar to those
  in Section~\ref{s:morereg2}, using the estimate
  in \cite[Lemma 3.3]{DebFou2013}.
  
  In conclusion, $\apperr\lesssim\epsilon^{\frac\theta\alpha(1+\kappa)}$,
  therefore, using the same computations of Proposition~\ref{p:bulk}
  and the smoothing Lemma~\ref{l:smoothing2},
  we finally obtain that $p _x(t)\in B^a_{p,\infty}(\R^d)$, for
  every $a<\kappa\bigl(1\wedge\frac\alpha{p}\bigr)$.
\end{proof}
\subsubsection{Rougher drift}\label{s:lrougher}

In this section we extend the results of
Section~\ref{s:rougherdrift} to the
$\alpha$-stable-like drivers. We consider
problem~\eqref{e:levy} with $b\in L^q(0,T;L^p(\R^d))$
(and same assumptions as before on $\sigma$).

As in Section~\ref{s:morereg2}, our result is an a--priori
estimate on the regularity of the density. In Section~\ref{s:morereg2}
the equation we consider has a unique strong solution \cite{KryRoc2005},
thus the result is rigorous. Here, under the assumptions on the coefficients we consider, we do not know if there
is a solution, or if it is a Markov process.
So, as in the above Section~\ref{s:lmorereg},
we assume that either there is a solution
of~\eqref{e:levy} that is a Markov
process, or that there is a solution that can be obtained
by approximation of \eqref{e:levy} with smooth coefficients.
In the latter case the Besov bound ensures the uniform integrability
of the approximating densities and thus that the limit problem
has a density with the same Besov regularity.
\begin{theorem}\label{t:lrougher}
  Let $(Z_t)_{t\geq0}$ be a Lévy process as in
  Assumption~\ref{a:levy}, with $\alpha>1$,
  let $b\in L^q(0,T;L^p(\R^d))$,
  $\sigma\in C^\beta_b(\R^d;\R^{d\times d'})$,
  such that \eqref{e:lnondegenerate} holds.
  For every $x\in\R^d$
  let $(X^x_t)_{t\geq0}$ be a solution (as specified
  above) of \eqref{e:levy}. Assume there is $e\geq0$
  such that
  \begin{equation}\label{e:lrougher}
    \begin{gathered}
      e q'
        <1,\qquad
      \alpha\kappa
        >1,\qquad
      \frac{d}{p}
        <\alpha\kappa-1,\qquad
      e
        > \frac{\kappa d}{p(\alpha\kappa-1)-d},
    \end{gathered}
  \end{equation}
  where $\kappa=\min(\frac1{q'},\frac{1+\beta}\alpha,
  \frac1\alpha+\frac\beta{q'}-\frac12\beta e)$,
  and $q'$ is the conjugate H\"older exponent of $q$.

  Then for every $t>0$ and $x\in\R^d$, the random
  variable $X_t^x$ has a density $p_x(t)$ in
  $B^b_{1,\infty}(\R^d)$ for every $b<\alpha\kappa-1$.
\end{theorem}
We recall that, if $p_x(t)$ is the density of a solution $X_t^x$
of \eqref{e:levy} with initial condition $x$, we have set
\[
  \|p_\cdot\|_{\star,\gamma}
    = \sup_{t\in(0,T]}\sup_{x\in\R^d}
      (1\wedge t)^{e_\gamma}\|p_x(t)\|_{B^\gamma_{1,\infty}},
\]
with $e_\gamma$ to be suitably chosen.
Set $a=\frac{d}{p}$, so that $p'=\frac{d}{d-a}$
and $B^a_{1,\infty}\subset L^{p'}$, where $p'$ is the
H\"older conjugate exponent of $p$.
 We start with
some estimates of the contribution of the drift.
\begin{lemma}\label{l:ldrift}
  If $q' e_a<1$ (for the second inequality), for $s\leq t$ and $\epsilon<1$, $\epsilon\leq\frac{t}{2}$,
  \[
    \begin{gathered}
      \E\Bigl[\int_{t-\epsilon}^s b(r,X_r)\,dr\Bigr]  
        \lesssim (1\wedge t)^{-e_a}\epsilon^{\frac1{q'}}
          \|b\|_{L^q(L^p)}\|p_\cdot\|_{\star,a},\\
      \E\Bigl[\Bigl(\int_{t-\epsilon}^s b(r,X_r)\,dr\Bigr)^2\Bigr]  
        \lesssim (1\wedge t)^{-e_a}\epsilon^{\frac2{q'}-e_a}
          \|b\|_{L^q(L^p)}^2\|p_\cdot\|_{\star,a}^2,
    \end{gathered}
  \]
  where $q'$ is the H\"older conjugate exponent of $q$.
\end{lemma}
\begin{proof}
  The first estimate follows as in Theorem~\ref{t:roughdrift},
  since we have,
  \[
    \E[b(r,X_s)]\lesssim (1\wedge s)^{-e_a}\|b(r)\|_{L^p}\|p_\cdot\|_{\star,a}.
  \]
  Likewise,
  \[
    \E\Bigl[\Bigl(\int_{t-\epsilon}^s b(r,X_r)\,dr\Bigr)^2\Bigr]
      = 2\int_{t-\epsilon}^s\,dr_1\int_{r_1}^t\E[b(r_1,X_{r_1})b(r_2,X_{r_2})],
  \]
  and, by the Markov property and the estimate above (twice),
  \[
    \begin{aligned}
      \E[b(r_1,X_{r_1})b(r_2,X_{r_2})]
        &= \E[b(r_1,X_{r_1})\E[b(r_2,X_{r_2-r_1}^y)]_{y=X_{r_1}}]\\
        &\lesssim (1\wedge (r_2-r_1))^{-e_a}\|b(r_1)\|_{L^p}\|p_\cdot\|_{\star,a}
          \E[b(r_1,X_{r_1})]\\
        &\lesssim \|b(r_1)\|_{L^p}\|b(r_2)\|_{L^p}\|p_\cdot\|_{\star,a}^2
          (1\wedge (r_2-r_1))^{-e_a}(1\wedge r_1)^{-e_a}.
    \end{aligned}
  \]
  Thus, by the H\"older inequality (twice),
  \[
    \begin{aligned}
      \lefteqn{\E\Bigl[\Bigl(\int_{t-\epsilon}^s b(r,X_r)\,dr\Bigr)^2\Bigr]\lesssim}\qquad\\
        &\lesssim \|b\|_{L^q(L^p)}^2\|p_\cdot\|_{\star,a}^2
          \Bigl(\int_{t-\epsilon}^t(1\wedge r_1)^{-e_a q'}
          \int_{r_1}^t(1\wedge (r_2-r_1))^{-e_a q'}\,dr_2\,dr_1
          \Bigr)^{\frac1{q'}}\\
        &\lesssim \|b\|_{L^q(L^p)}^2\|p_\cdot\|_{\star,a}^2
          (1\wedge t)^{-e_a}\epsilon^{\frac2{q'}-e_a},
    \end{aligned}
  \]
  by elementary computations, since $q' e_a<1$.
\end{proof}
From the above estimates we immediately deduce the following
result.
\begin{lemma}\label{l:ldiff}
  If $\gamma<\alpha$ and $s\in[t-\epsilon,t]$,
  \[
    \E[|X_s-X_{t-\epsilon}|^\gamma]^{\frac1\gamma}
      \lesssim (1+\|\sigma\|_{L^\infty}+\|b\|_{L^q(L^p)})
        (1\vee\|p_\cdot\|_{\star,a})(1\wedge t)^{-\frac{e_a}2}\epsilon^{\kappa_1},
  \]
  where $\kappa_1=\min(\frac1\alpha,\frac1{q'}-\frac{e_a}2)$
\end{lemma}
\begin{proof}
  We have that
  \[
    X_s - X_{t-\epsilon}
      = \int_{t-\epsilon}^s b(r,X_r)\,dr
        + \int_{t-\epsilon}^s \sigma(X_{r-})\,dZ_r,
  \]
  thus
  \[
    \E[|X_s-X_{t-\epsilon}|^\gamma]
      \lesssim \E\Bigl[\Bigl|\int_{t-\epsilon}^s b(r,X_r)\,dr\Bigr|^\gamma\Bigr]
        + \E\Bigl[\Bigl|\int_{t-\epsilon}^s \sigma(X_{r-})\,dZ_r\Bigr|^\gamma\Bigr]
      = \memo{D} + \memo{S}.
  \]
  By \cite[Lemma A.2-(i)]{DebFou2013},
  \[
    \memo{S}
      \lesssim (s-(t-\epsilon))^{\gamma/\alpha}\|\sigma\|_{L^\infty}^\gamma
      \leq \epsilon^{\gamma/\alpha}\|\sigma\|_{L^\infty}^\gamma
  \]
  For \memo{D} we use the previous lemma, since
  \[
    \memo{D}
      \leq \E\Bigl[\Bigl|\int_{t-\epsilon}^s b(r,X_r)\,dr\Bigr|^2\Bigr]^{\frac\gamma2}
      \lesssim \|b\|_{L^q(L^p)}^\gamma\|p_\cdot\|_{\star,a}^\gamma
        (1\wedge t)^{-e_a\gamma/2}\epsilon^{\frac\gamma2(\frac2{q'}-e_a)}.\qedhere
  \]
\end{proof}
\begin{proof}[Proof of Theorem~\ref{t:lrougher}]
  We are ready to estimate the \Ae and the
  \Pe. We use the auxiliary process
  as in formula~\eqref{e:laux} of previous section.
  The \Pe can be immediately deduced from
  Lemma~\ref{l:lpe} (with $p=1$), thus
  $\prober\lesssim \epsilon^{-m/\alpha}|h|^m\|\phi\|_{L^\infty}$.
  We turn to the \Ae. We have,
  \[
    \begin{aligned}
      \apperr
        &\lesssim [\phi]_{\Cs^\theta}\E[|X_t-Y_t^\epsilon|^\theta]\\
        &\lesssim [\phi]_{\Cs^\theta}\Bigl(
          \E\Bigl|\int_{t-\epsilon}^t b(s,X_s)\,ds\Bigr|^\theta
          + \E\Bigl|\int_{t-\epsilon}^t (\sigma(X_{s-})-\sigma(X_{t-\epsilon})\,dZ_s\Bigr|^\theta
          \Bigr).
    \end{aligned}
  \]
  For the first term we use the first statement of
  Lemma~\ref{l:ldrift} (recall that $\theta<1$),
  \[
    \E\Bigl|\int_{t-\epsilon}^t b(s,X_s)\,ds\Bigr|^\theta
      \leq\E\Bigl[\Bigl|\int_{t-\epsilon}^t b(s,X_s)\,ds\Bigr|\Bigr]^\theta
      \lesssim (1\wedge t)^{-e_a\theta}\epsilon^{\frac\theta{q'}}
        \|p_\cdot\|_{\star,a}^\theta.
  \]
  For the second term we use \cite[Lemma A.2-(i)]{DebFou2013} and
  Lemma~\ref{l:ldiff}: let $\gamma$ be such that $\alpha<\gamma\leq2$
  and $\gamma\beta<\alpha$, then
  \[
    \begin{aligned}
      \E\Bigl|\int_{t-\epsilon}^t (\sigma(X_{s-})-\sigma(X_{t-\epsilon})\,dZ_s\Bigr|^\theta
        &\lesssim \epsilon^{\theta/\alpha}\sup_{[t-\epsilon,t]}
          \E[|\sigma(X_{s-})-\sigma(X_{t-\epsilon})|^\gamma]^{\frac\theta\gamma}\\
        &\lesssim \epsilon^{\theta/\alpha}[\sigma]_{\Cs^\beta}^\theta
          \sup_{[t-\epsilon,t]}\E[|X_s-X_{t-\epsilon}|^{\gamma\beta}]^{\frac\theta\gamma}\\
        &\lesssim (1\wedge t)^{-\frac12 \beta\theta e_a}
          (1\vee\|p_\cdot\|_{\star,a})^{\beta\theta}
          \epsilon^{\theta(\kappa_1\beta+\frac1\alpha)}
    \end{aligned}
  \]
  If we put the two estimates together,
  \[
    \apperr
      \lesssim (1\wedge t)^{-\theta e_a}
        (1\vee\|p_\cdot\|_{\star,a})^\theta
        \epsilon^{\theta\kappa}[\phi]_{\Cs^\theta_b},      
  \]
  where $\kappa=\min(\frac1{q'},\frac{1+\beta}\alpha,
  \frac1\alpha+\frac\beta{q'}-\frac12\beta e_a)$.
  With the positions
  $K_0=(1\wedge t)^{-\alpha e_a}(1\wedge\|p_\cdot\|_{\star,a})^\alpha$
  and $a_0=\alpha\kappa-1$, Proposition~\ref{p:bulk}
  yields that
  \[
    \|p_x(t)\|_{B^a_{1,\infty}}
      \lesssim (1\wedge\|p_\cdot\|_{\star,a})^{\frac{a}{a_0}+\alpha\delta}
        (1\wedge t)^{-\frac{a}{a_0}e_a-\alpha\delta e_a-\frac{1+a_0}{\alpha a_0}a-\delta},
  \]
  All the above computations show that under the following conditions,
  \[
    \begin{gathered}
      q' e_a
        <1,\qquad
      \alpha\kappa
        >1,\qquad
      a=\frac{d}{p}
        <a_0,\\
      e_a
        >\frac{a}{a_0}e_a + \frac{1+a_0}{\alpha a_0}a,
    \end{gathered}
  \]
  the norm $\|p_\cdot\|_{\star,a}$ is finite, and thus
  $\|p_\cdot\|_{\star,b}<\infty$ for every $b<a_0$,
  for a suitable value of $e_b$ that can be
  easily computed as above by Proposition~\ref{p:bulk}.
\end{proof}
It is not immediately apparent that the conditions of
Theorem~\ref{t:lrougher} may be verified,
so we provide a couple of particular cases.
The first matches Theorem~\ref{t:roughdrift}. Notice that
under slightly different assumptions on the drift $b$
but that are essentially the same as those in the corollary
below, existence and uniqueness hold, see \cite{CheWan2016}.
\begin{corollary}
  Under the assumptions of Theorem~\ref{t:lrougher}
  on the solution and on the coefficients, and if moreover
  $\sigma$ is constant, then $\kappa=\frac1{q'}$ and one can
  take $e=0$. In particular, the conclusions of the theorem
  hold if
  \[
    \frac{d}{p} + \frac{\alpha}{q}
      < \alpha-1.
  \]
\end{corollary}
The second particular case applies for $p$ large.
In that case we expect the number $e$ of Theorem~\ref{t:lrougher}
to be small. As a matter of facts,
$e\to0$ as $p\to\infty$.
To further simplify, we consider the case $\kappa=\frac1{q'}$.
\begin{corollary}
  Under the same assumptions of Theorem~\ref{t:lrougher}, if
  \begin{itemize}
    \item $\frac{2d}{p} + \frac{\alpha}{q}<\alpha-1$,
    \item $q\geq\frac{\alpha}{2-\alpha}$,
    \item either $\beta\geq\alpha-1$, or $\beta<\alpha-1$
      and $q\leq\frac{\alpha}{\alpha-\beta-1}$,
  \end{itemize}
  then any number $e$ such that
  \[
    \frac{\kappa d}{p(\alpha\kappa-1) - d}
      < e
      < \frac1{q'}
  \]
  meets the conditions \eqref{e:lrougher}, and thus
  the conclusions of Theorem~\ref{t:lrougher} hold.
\end{corollary}
\begin{proof}
  Let us first prove that $\kappa=\frac1{q'}$
  if $eq'<1$.
  Indeed, by the third of the assumptions above,
  $\frac1{q'}\leq\frac{\beta+1}{\alpha}$.
  Moreover, by the second of the assumptions
  above, $\frac1{q'}\leq2(\frac1{q'}-\frac1\alpha)$,
  thus if $eq'<1$, then
  $e<2(\frac1{q'}-\frac1\alpha)$, therefore
  $\frac{1+\beta}\alpha\leq\frac1\alpha+\frac\beta{q'}
  -\frac12\beta e$, and in conclusion
  $\kappa=\frac1{q'}$.
  
  The condition
  $\alpha\kappa>1$ holds since $\beta>0$ and $q'<\alpha$.
  The condition $\frac{d}{p}<\alpha\kappa-1$ holds
  by the first of the assumptions of the corollary.
  To ensure that $e$ can be chosen in a non--empty
  interval, we need to check that
  \[
    \frac{\kappa d}{p(\alpha\kappa-1) - d}
      <\frac1{q'},
  \]
  and this follows from the first assumption
  of the corollary.
\end{proof}
\subsection{The 3D Navier--Stokes equations with noise}

The problem of existence of densities for finite dimensional
projections of the solutions of the Navier-Stokes equations
driven by noise has been the motivating example,
discussed in \cite{DebRom2014}, that
has led to the development of the dimension-free method
illustrated in this paper.
The results have been further improved
in \cite{Rom2016b} proving H\"older regularity in time
with values in Besov spaces in time, and in \cite{Rom2016}
proving H\"older regularity in space of the densities.
Unlike Section~\ref{s:morereg2}, the result of H\"older
regularity is optimal but has been proved by analytical
methods.
\subsection{A singular equation: \texorpdfstring{$\Phi^4_d$}{Phi4d}}

In this section we consider the following
singular stochastic PDE,
\begin{equation}\label{e:phi4d}
  \partial_t X
    = \Delta X + X-X^3 + \xi,
\end{equation}
on the torus $\Tb_d$, with periodic boundary conditions,
in two dimensions (but see Remarks~\ref{r:phi431} and
\ref{r:phi432} for the three dimensional case),
where $\xi$ is space-time white noise.

The equation is generally understood
as the limit of a family of regularized
problems.
To this end, let $\eta$ be a smooth compactly
supported function, set
$\eta_\delta(x)=\delta^{-d}\eta(\delta^{-1}x)$,
and $\xi_\delta=\eta_\delta\star\xi$. Consider
\begin{equation}\label{e:phi4dreg}
  \partial_t X_\delta - \Delta X_\delta
    = - X_\delta^3 + (1+3c_\delta)X_\delta + \xi_\delta,
\end{equation}
where $c_\delta$ is a suitable number that depends
on $\delta$ and $\eta$. Define
\[
  (\partial_t - \Delta)\ga_\delta
    = -\ga_\delta + \xi_\delta,
      \qquad
  \gb_\delta
    = \ga_\delta^2 - c_\delta,
      \qquad
  \gc_\delta
    = \ga_\delta^3 - 3c_\delta\ga_\delta.
\]
where $\ga$ is chosen as the stationary process
that solves the above equation. With this choice,
$\gb_\delta$ and $\gc_\delta$ are also
stationary processes.
The constants $c_\delta$
are chosen so that $\gb_\delta$ and $\gc_\delta$
have a limit as $\delta\downarrow0$, see
\cite{DapDeb2003,MouWeb2015} in dimension two,
and \cite{Hai2014,Hai2015,Hai2016,CatCho2013,MouWeb2016}
for the three dimensional case.
We will denote by $\ga,\gb,\gc$ the limit
quantities as $\delta\downarrow0$.
\subsubsection{Densities for the solution}

Set $X_\delta=\ga_\delta + R_\delta$,
then $R_\delta$ solves
\[
  (\partial_t-\Delta)R_\delta
    = - R_\delta^3 - 3\ga_\delta R_\delta^2
      + (1 - 3\gb_\delta)R_\delta
      + 2\ga_\delta - \gc_\delta.
\]
We prove that the solution has a density at
each time. Since the solution is distribution valued,
we will prove the existence of a joint density of
the solution tested over a finite but arbitrary number
of test functions.
\begin{theorem}\label{t:phi4d1}
  Let $X$ be a solution of \eqref{e:phi4d} on the
  two dimensional torus, let
  $t>0$, an integer $n\geq1$ and let
  $\varphi_1,\varphi_2,\dots,\varphi_n$
  be smooth periodic functions such that the matrix
  $(\scalar{\varphi_i,\varphi_j})_{i,j=1,\dots,n}$
  is non--singular.
  
  Then the random vector $(\scalar{X(t),\varphi_1},
  \scalar{X(t),\varphi_1},\dots,\scalar{X(t),\varphi_n})$
  has a density with respect to the Lebesgue measure
  in $B^a_{1,\infty}(\R^n)$, for every $a<1$.
\end{theorem}
\begin{proof}
  We will perform our estimates on the solution
  of~\eqref{e:phi4dreg}. The Besov bound of the
  densities will ensure the existence of a limit
  density.
  
  Equation~\eqref{e:phi4dreg} in mild form and
  evaluated on each test function $\varphi_i$,
  $i=1,2,\dots,n$, yields the evolution
  \[
    dX_{\delta,i}
      = \scalar{\Delta X_\delta + (1+3c_\delta)X_\delta
        -X_\delta^3,\varphi_i} + \scalar{\varphi_i,\xi_\delta},
  \]
  of the random vector
  $(X_{\delta,1},X_{\delta,2},\dots,X_{\delta,n})$,
  defined as $X_{\delta,i}(t)=\scalar{X_\delta(t),\varphi_i}$. 
  
  Fix $t>0$, then the auxiliary process for the simple
  method (the term analogous to \ref{e:auxiliary}) here
  is given as
  $dY_{\delta,i}^\epsilon=\scalar{\varphi_i,\xi_\delta}$.
  By the non--degeneracy assumption on the test functions
  it follows that, for $\delta$ small enough, the random vector
  $(\scalar{\varphi_1,\xi_\delta},\scalar{\varphi_2,\xi_\delta},
  \dots,\scalar{\varphi_n,\xi_\delta})_{i,j=1,2,\dots,n}$
  is a non--degenerate Gaussian random vector. It easily
  follows then that the \Pe is given by
  $\prober\lesssim\epsilon^{-m/2}|h|^m\|\phi\|_{L^\infty}$.
 
  For the \Ae we need to compute
  \[
    \E\int_{t-\epsilon}^t\scalar{\Delta X_\delta
      + (1+3c_\delta)X_\delta - X_\delta^3,\varphi_i}\,ds.
  \]
  This is immediate, since $X_\delta$ has moments in (negative)
  Besov spaces \cite{MouWeb2015}. This is read in terms
  of uniform bounds on moments of $X_\delta$ in negative
  Besov spaces. In conclusion
  $\apperr\lesssim \epsilon^\alpha[\phi]_{\Cs^\alpha_b}$.
  Proposition~\ref{p:bulk} concludes the proof.
\end{proof}
\begin{remark}
  We actually expect the densities in the previous
  proposition to be smooth. Following the lines of
  Section~\ref{s:morereg1}, one could define a
  infinite dimensional auxiliary process $Y^\epsilon$
  with the drift frozen at $t-\epsilon$, as in formula
  \eqref{e:aux2}. This would allow to compute an \Ae
  at the level of the infinite dimensional processes
  in the appropriate Besov spaces of distributions.
  When evaluated over test functions, this error
  would provide the \Ae for the method.
\end{remark}
\begin{remark}\label{r:phi431}
  The above result holds also in dimension three,
  with the same proof. The difference is that,
  as we shall see below, the correct interpretation
  of the equation is more involved.

  In dimension three we additionally introduce the
  diagram $\goc_\delta$, the stationary solution of
  $(\partial_t - \Delta)\goc_\delta = \gc_\delta$.
  Again, the constant $c_\delta$
  is chosen so that $\gb_\delta$ and $\gc_\delta$
  have a limit as $\delta\downarrow0$, see
  \cite{MouWeb2016}.
  Set $X_\delta=\ga_\delta-\goc_\delta+R_\delta$,
  then $R_\delta$ solves
  \[
    (\partial_t-\Delta)R_\delta
      = -R_\delta^3-3(R_\delta-\goc_\delta)\gb_\delta
        + P_\delta(R_\delta),
  \]
  where
  \[
    P_\delta(R_\delta)
      = 3(\goc_\delta-\ga_\delta)R_\delta^2
        + (1+6\ga_\delta\goc_\delta-3\goc_\delta^2)R_\delta
        + (\ga_\delta-\goc_\delta)+\goc_\delta^3
          - 3\ga_\delta\goc_\delta^2.
  \]
  This is not enough yet to give a meaning to the
  limit equation, because the term $R_\delta\gb_\delta$
  is not well defined in the limit, given the expected
  regularity of the limit $R$. Without giving too many
  details (that would be beyond the scope of this paper),
  we know that the equation can be suitably
  reformulated following for instance \cite{MouWeb2016}
  (see also \cite{CatCho2013}) as,
  \[
    (\partial_t - \Delta)R_\delta
      = -R_\delta^3-3R^H_\delta\peq\gb_\delta
        -3R_\delta\plt\gb_\delta+Q_\delta(R_\delta)
        + 3\gc_\delta\plt\gb_\delta
  \]
  where $R_\delta=R^L_\delta+R^H_\delta$, $R^L_\delta$
  is solution of
  \[
    (\partial_t-\Delta)R^L_\delta
      = -3(R_\delta-\gc_\delta)\plt\gb_\delta,
  \]
  $Q_\delta$ gathers more regular terms,
  and $\peq$, $\pgt$ are defined in terms of the
  Bony paraproduct.
  
  At this stage one can proceed as in the
  two-dimensional case, since the stochastic diagrams
  as well as the remainder have uniform bounds in time
  on moments in negative Besov spaces. Thus,
  Theorem~\ref{t:phi4d1} holds for \eqref{e:phi4d}
  also in dimension $d=3$.
\end{remark}
\begin{corollary}\label{c:phi4d1}
  Let $X$ be a solution of \eqref{e:phi4d} on the
  three dimensional torus, let
  $t>0$, an integer $n\geq1$ and let
  $\varphi_1,\varphi_2,\dots,\varphi_n$
  be smooth periodic functions such that the matrix
  $(\scalar{\varphi_i,\varphi_j})_{i,j=1,\dots,n}$
  is non--singular.
  
  Then the random vector $(\scalar{X(t),\varphi_1},
  \scalar{X(t),\varphi_1},\dots,\scalar{X(t),\varphi_n})$
  has a density with respect to the Lebesgue measure
  in $B^a_{1,\infty}(\R^n)$, for every $a<1$.
\end{corollary}
\subsection{Densities for the remainder}

In this section we wish to delve into another direction.
As a second application of the simple method in this framework,
we wish to investigate the existence of
a density for the remainder $R$.
A possible approach could be based on the previous
considerations and some arguments from
Section~\ref{s:hypo}. Indeed, it would be
sufficient to prove (for instance in
dimension two) that $(X,\ga)$ has
a joint density (when evaluated over
test functions). Since both $X$ and $\ga$
are driven by the same noise, this can be only
possible if the drift is hypoelliptic.
This idea has two drawbacks: the first is
that it would give densities for $R$ against
test functions, while $R$ is a \emph{bona fide}
function. The second is that, as we have seen
in Section~\ref{s:hypo}, the method requires
good estimates on the drift, while giving
back low regularity.

We wish to consider here a different idea,
that uses the non--linearity directly. The equation
for $R$ is (in dimension two),
\[
  (\partial_t-\Delta)R
    = - R^3 - 3\ga R^2
      + (1 - 3\gb)R
      + 2\ga - \gc.
\]
Our idea is to
understand the terms not depending
on $R$ as the ``noise'' of the evolution,
since our simple method
presented in this paper
is essentially based on small time estimates
for the density of the ``noise'' object.

Unfortunately the idea does not work here
in dimension two (and apparently in dimension
three as well, see Remark~\ref{r:phi432}),
since $\ga$ and $\gc$ are
``too good'', that is with regularity comparable
with the ``non--noise'' terms
$R^3 + 3\ga R^2 - (1 - 3\gb)R$. As such,
we would end up with a formula
similar to~\eqref{e:bulkstart}, but with
$a_0=0$. For this reason
in the rest of the section
we will discuss an example
very close to problem~\eqref{e:phi4d}
and that is amenable to our
analysis.

We consider the following problem,
\begin{equation}\label{e:phi4alpha}
  (\partial_t-\Delta)X
    = -X^3 + (-\Delta)^{\gamma/2}\xi,
\end{equation}
with periodic boundary conditions
on the two dimensional torus, where
$\gamma\in(0,\frac25)$ and $\xi$
is space-time white noise. As above, the problem makes
sense when suitably renormalized or as a limit of
approximated problems. If $\xi_\delta$ is a spatial
smooth approximation of the noise, we study the
problem
\[
  (\partial_t-\Delta)X_\delta
    = -X_\delta^3 + 3c_{\delta}(t)X_\delta + (-\Delta)^{\gamma/2}\xi_\delta,
\]
where $c_\delta(t)=\E[\ga_\delta(t)^2]$, and $\ga_\delta$
solves $(\partial_t-\Delta)\ga_\delta=(-\Delta)^{\gamma/2}\xi_\delta$.
Set $\gb_\delta(t)=\ga_\delta(t)^2 - c_\delta(t)$ and
$\gc_\delta(t)=\ga_\delta(t)^3 - 3c_\delta(t)\ga_\delta(t)$.
Denote respectively by $\ga$, $\gb$ and $\gc$ their
limits as $\delta\to\infty$. From now on we will drop
for simplicity
the subscript $\delta$, even though all computations
are rigorous only at the level of approximations.
The Besov bound from our method will provide the
uniform integrability necessary to bring the
argument rigorously in the limit as $\delta\to0$
to the solution of \eqref{e:phi4alpha}.

The choice $\gamma<\frac25$ is for convenience.
In this regime it is sufficient to decompose
the solution as $X=\ga+R$, and the remainder
solves
\[
  (\partial_t-\Delta)R
    = -R^3-3\ga R^2-3\gb R-\gc.
\]
One can prove, for instance following the lines
of \cite[Theorem 1.1]{MouWebXu2016},
that $\ga$ has regularity
$\Cs^{-\gamma-}$, $\gb$ has
regularity $\Cs^{-2\gamma-}$ and
$\gc$ has regularity $\Cs^{-3\gamma-}$.
Thus we expect that $R$ has regularity
$\Cs^{(2-3\gamma)-}$, so that
$\ga R^2$ and $\gb R$ are well defined
when $\gamma<\frac25$\footnote{If $\gamma$ is
above $\frac25$ but below $\frac12$, we
need to decompose $X$ with the additional
term $\goc$. The case $\gamma=\frac12$
is, analytically, equivalent to $\Phi^4_3$.}.
In the rest of the section we will
\emph{assume} that we have a solution
for the auxiliary equation with the
above mentioned regularity.
In principle the solution might be
defined only up to a random time,
but on the one hand we may guess
that the results of~\cite{MouWeb2017}
extend to this case, and on the other
hand our method can take local solutions
into account as well (as in
Section~\ref{s:blowup}).

The main result for the densities of
the remainder is as follows.
\begin{theorem}\label{t:phi4d2}
  Let $X$ be a solution of \eqref{e:phi4alpha},
  and set $R_t=X_t - \ga_t$. Then
  for every $t>0$ and every $x\in\Tb_2$
  the random variable $R_t(x)$ has
  a density with respect to the
  Lebesgue measure on $\R$. Moreover
  the density is in $B^a_{1,\infty}$
  for every $a<\frac{\gamma}{2-3\gamma}$.
\end{theorem}
To prove the theorem we first identify the
approximation problem, as in formula
\eqref{e:auxiliary} for the toy problem.
Fix $0<s<t$ and write the equation for $R$ in mild form,
\[
  R_t
    = \e^{\Delta(t-s)}R_s
      - \int_s^t \e^{\Delta(t-r)}(R_r^3+3\ga_r R_r^2+3\gb_rR_r)\,dr
      - \int_s^t \e^{\Delta(t-r)}\gc_r\,dr.
\]
The terms in the first integral in the formula above
are more regular than the second integral, thus
should provide a smaller \Ae. We will treat the
last term as ``noise'', although clearly this
new ``noise'' has no independent increments.
We define our auxiliary process
for $s\leq r\leq t$, as
\begin{equation}\label{e:phi4daux}
  S_r
    = \e^{\Delta(r-s)}R_s
      - \int_s^r\e^{\Delta(r-u)}\gc_u\,du.
\end{equation}
\begin{lemma}\label{l:phi4dae}
  Under the standing assumptions, given $\phi\in\Cs^\alpha$
  for some $\alpha\in(0,1)$,
  \[
    \apperr
      \approx\E[\phi(R_t(x))] - \E[\phi[S_t(x)]]
      \lesssim [\phi]_{\Cs^\alpha}(t-s)^{1-\gamma-\delta},
  \]
  for every $0\leq s<t$, $x\in\Tb_2$, and every $\delta>0$.
\end{lemma}
\begin{proof}
  We have that
  \begin{equation}\label{e:phi4diff}
    S_t - R_t
      = \int_s^t\e^{\Delta(t-r)}(R_r^3+3\ga_r R_r^2+3\gb_r R_r)\,dr
  \end{equation}
  Among the terms in the right-hand-side of formula
  above, the least regular is the one containing
  $\gb_r R_r$. By \cite[Proposition 2.3]{MouWebXu2016},
  \[
    \|\gb_r R_r\|_{\Cs^{-2\gamma-\delta}}
      \lesssim \|\gb_r\|_{\Cs^{-2\gamma-\delta}}\|R_r\|_{\Cs^{2-3\gamma-\delta}},
  \]
  with $\delta$ such that $0<\delta<1-\frac52\gamma$. Therefore,
  \[
    \begin{aligned}
      \Bigl|\Bigl(\int_s^t\e^{\Delta(t-r)}\gb_r R_r\,dr\Bigr)(x)\Bigr|
        &\lesssim\Bigl(\int_s^t(t-r)^{-\gamma-\delta}\,dr\Bigr)
          \bigl(\sup_{[0,t]}\|R\|_{\Cs^{2-3\gamma-\delta}}\bigr)
          \bigl(\sup_{[0,t]}\|\gb\|_{\Cs^{-2\gamma-\delta}}\bigr)\\
        &\lesssim (t-s)^{1-\gamma-\delta}.
    \end{aligned}
  \]
  The terms containing $R_r^3$ and $\ga_r R_r^2$ are more regular
  and give a smaller contribution in terms of powers of $t-s$.
\end{proof}
We wish now to ``extract'' the component of
our ``noise'' that is independent from the
past (that is before time $s$).
To this end notice that if $s<r<t$, then
\[
  \ga_r
    = \e^{\Delta(r-s)}\ga_s
      + \int_s^r\e^{\Delta(r-u)}\,dW_u
    \qedef \ga_{s\to r} + \ga_{s,r},
\]
with $\ga_{s\to r}$ measurable with respect to
the history up to time $s$ (and smooth for $r>s$),
and $\ga_{s,r}$ independent from the history
up to time $s$. By squaring,
\[
  \gb_r
    = \ga_r^2 - \E[\ga_r^2]
    = \bigl(\ga_{s\to r}^2 - \E[\ga_{s\to r}^2]\bigr)
      + 2\ga_{s\to r}\ga_{s,r}
      + \bigl(\ga_{s,r}^2 - \E[\ga_{s,r}^2]\bigr),
\]
and we set
\[
  \gb_{s,r}
    = \ga_{s,r}^2 - \E[\ga_{s,r}^2],
      \qquad
  \gb_{s\to r}
    = \ga_{s\to r}^2 - \E[\ga_{s\to r}^2].
\]
Likewise,
\[
  \begin{aligned}
    \gc_r
      &= \ga_r^3 - 3\E[\ga_r^2]\ga_r\\
      &= \bigl(\ga_{s\to r}^3 - 3\E[\ga_{s\to r}^2]\ga_{s\to r}\bigr)
        + 3\gb_{s\to r}\ga_{s,r}
        + 3\ga_{s\to r}\gb_{s,r}
        + \bigl(\ga_{s,r}^3 - 3\ga_{s,r}\E[\ga_{s,r}^2]\bigr),
  \end{aligned}
\]
and we set
\[
  \gc_{s,r}
    = \ga_{s,r}^3 - 3\ga_{s,r}\E[\ga_{s,r}^2],
      \qquad
  \gc_{s\to r}
    = \ga_{s\to r}^3 - 3\E[\ga_{s\to r}^2]\ga_{s\to r}.
\]
In conclusion,
\[
  S_t
    = \e^{\Delta(t-s)}R_s
      - \int_s^t\e^{\Delta(t-r)}\bigl(\gc_{s\to r}
        + 3\gb_{s\to r}\ga_{s,r}+3\ga_{s\to r}\gb_{s,r}
        + \gc_{s,r}\bigr)\,dr.
\]
\begin{lemma}\label{l:phi4dpe}
  Under the standing assumptions, given $\phi\in L^\infty$,
  \[
    \prober
      = \E[\Delta_h^m\phi(S_t(x))]
      \lesssim \|\phi\|_{L^\infty}\frac{|h|^m}{(t-s)^{m(1-\frac32\gamma+\delta)}},
  \]
  for every $0\leq s<t$, $x\in\Tb_2$, $m\geq 1$, $h\in\Tb_2$
  with $|h|\leq 1$, and every $\delta>0$.
\end{lemma}
\begin{proof}
  Since
  \begin{equation}\label{e:phi4dpestima}
    \E[\Delta_h^m\phi(S_t(x))]
      = \E\bigl[\E[\Delta_h^m\phi(S_t(x))\,|\Fs_s]\bigr],
  \end{equation}
  where $\Fs_s$ is the $\sigma$--field of events
  until time $s$, the result follows if we can
  estimate the small time asymptotic of the
  $L^1$ norm of the density (and of its
  derivatives) of $S_t(x)$ given $\Fs_s$.

  To this end we recall that for a real random
  variable $X$ with density $g$, we have by
  \cite[Theorem 2.1.4]{Nua2006}, the representation
  \[
    D^mg(x)
      = \E[\uno_{\{X>x\}}H_{(1,1,\dots,1)}],
  \]
  (there are $(m+1)$ ones in the subscript above)
  for the density and its derivatives ($m\geq0$), as well as
  \[
    D^mg(x)
      = \E[\uno_{\{X<x\}}H_{(1,1,\dots,1)}],
  \]  
  where the terms $H_\cdot$ come from integration by
  parts and are defined in \cite[Proposition 2.1.4]{Nua2006}.
  Moreover, by \cite[formula~(2.28)]{Nua2006},
  \[
    \E[|H_{(1,1,\dots,1)}|^p]^{\frac1p}
      \lesssim\|\Mcc_X^{-1}\Dc X\|_{m+1,2^m r}^{m+1},
  \]
  for every $r>p$, where $\Dc$ is the Malliavin
  derivative, and $\Mcc_X$ is the Malliavin
  matrix of $X$. By integrating over $x$, we see
  that
  \begin{equation}\label{e:phi4ddiff}
    \|D^mg\|_{L^1}
      \lesssim \E[|X|\,|H_{(1,1,\dots,1)}|],
  \end{equation}
  so that our task is to understand how
  this expectation scales in powers of
  $(t-s)$ when we take
  \[
    X
      = -\int_0^{t-s}\e^{\Delta(t-s-r)}\bigl(U_3(r)
        + 3U_2(r)\ga_r+3U_1(r)\gb_r
        + \gc_r\bigr)\,dr,
  \]
  and $U_i$, $i=0,\dots,3$ are given elements
  with $U_i$ uniformly bounded in time with
  values in $\Cs^{-i\cdot\gamma-\delta}$
  for $i=1,2,3$, for every $\delta>0$.
  Notice that the choice of $X$ above
  corresponds, up to an additive constant,
  to $S_t$ when conditioned
  over $\Fs_s$, and with a translation in
  time of the stochastic diagrams (that
  does not change the law).
  By the chain rule for the Malliavin
  derivative,
  \[
    \Dc X
      = - \int_0^{t-s}\e^{\Delta(t-s-r)}\bigl(
            3 U_2(r)\Dc\ga_r+6U_1(r)\ga_r\Dc\ga_r
            + 3\gb_r\Dc\ga_r\bigr)\,dr,
  \]
  and similarly for the second and third derivative
  (the fourth derivative is zero since the random
  variable $X$ above is in the third Wiener chaos).
  In particular,
  \[
    \Dc_u\ga_r
      = \uno_{\{u\leq r\}}\e^{\Delta(r-u)}
        \sum |k|^\gamma e_k,
  \]
  where $(e_k)_{k\in\Z^2}$ is the Fourier basis
  of complex exponentials.
  
  We see that for every small $\delta>0$,
  \[
    \E\Bigl[\Big|\Bigl(\int_0^{t-s}
        \e^{\Delta(t-s-r)}\gc_r
        \Bigr)(x)\Big|^2\Bigr]
      \approx(t-s)^{2-3\gamma-2\delta},
  \]
  and likewise,
  \[
    \E\Bigl[\Big|\Bigl(\int_0^{t-s}
        \e^{\Delta(t-s-r)}U_1(r)\gb_r
        \Bigr)(x)\Big|^2\Bigr]
      \approx(t-s)^{2-2\gamma-2\delta}
        \sup_{[0,t-s]}\|U_1\|_{\Cs^{2\gamma+\delta}}.
  \]
  Notice that when we will
  evaluate the external expectation
  in \eqref{e:phi4dpestima}, we will see
  that $\sup_{[0,t-s]}\|U_1\|_{\Cs^{2\gamma+\delta}}
  \approx(t-s)^{-\frac\gamma2}$ (since we will
  replace $U_1(r)$ by $\ga_{s\to r}$).
  In conclusion $X$ scales as $(t-s)^{1-\frac32\gamma}$,
  up to small corrections.
  With similar computations we see that
  $\Dc X$ scales also as $(t-s)^{1-\frac32\gamma}$
  (due to the additional contribution of
  $\Dc\ga$), as well as the second and third Malliavin
  derivatives. Therefore, using formula~\eqref{e:phi4ddiff},
  the conclusion of the lemma follows.
\end{proof}
We are ready to complete the proof of the main result
of this section.
\begin{proof}[Proof of Theorem~\ref{t:phi4d2}]
  In view of Lemma~\ref{l:phi4dae} and
  Lemma~\ref{l:phi4dpe}, with the choice
  $s=t-\epsilon$, this is a
  simple application of
  Proposition~\ref{p:bulk}, with
  $\theta=1/(1-\tfrac32\gamma+\delta)$
  and $a_0=(1-\gamma)\theta-1$.
\end{proof}
\begin{remark}\label{r:phi432}
  One can expect that Theorem~\ref{t:phi4d2}
  might hold also for \eqref{e:phi4d} in dimension
  three. In that case the role of the ``noise''
  should be played by
  \[
    \int_s^t\e^{\Delta(t-r)}\gc_r\plt\gb_r\,dr.
  \]
  Since the remainder has regularity
  $\Cs^{1-}$, the \Ae should be of
  the order
  $\apperr\approx\epsilon^{\frac12-}$,
  while the \Pe has a density with
  increments of order $|h|/\sqrt{\epsilon}$.
  This prevents the application of
  Proposition~\ref{p:bulk} (in short, it
  would correspond to the case $a_0=0$).
  Otherwise
  one could improve the \apperr with
  a smarter definition of the
  auxiliary process, as in
  Section~\ref{s:morereg1}.
  In principle one could try
  to use H\"older continuity in time
  of the stochastic diagrams.
  Unfortunately this would
  deteriorate the
  space regularity (see for
  instance \cite[Theorem 1.1]{MouWebXu2016}),
  yielding a final dependence of
  the \Ae of order
  $\apperr\approx\epsilon^{1/2-}$
  (or worse).
  We do not try this attempt
  here though.
\end{remark}
\appendix
\section{Weaker versions of the smoothing lemma}\label{s:smoothing}

In the first pages of Malliavin's seminal paper
\cite{Mal1978} on a probabilistic proof of the
H\"ormander theorem there is
a classical smoothing lemma. This is the link between
the existence of a density and probabilistic integration
by parts and the Malliavin calculus.
The lemma says roughly that if for a $\R^d$-valued
random variable $X$,
\[
  \E[D^\alpha\phi(X)]
    \leq c_n\|\phi\|_{L^\infty},
\]
for all $|\alpha|\leq n$ and all test functions $\phi$,
then $X$ has a density, with respect to the Lebesgue
measure, in $C^{n-d-1}$. In this section we give a
generalization of this lemma in Besov spaces.
\subsection{Besov and Triebel--Lizorkin spaces}\label{s:primer}

Besov spaces, together with the Triebel--Lizorkin spaces,
are a scale of function spaces introduced to capture the
fine properties of regularity of functions, beyond on the
one hand the Sobolev spaces, and on the other hand the
spaces of continuous functions. Indeed, Besov spaces
contain both. The main references we shall use on this
subject are \cite{Tri1983,Tri1992}.

A general definition with the Littlewood--Paley decomposition
is (briefly) as follows. Let $(\varphi_n)_{n\geq0}$ be a
band--limited decomposition of the frequency space.
For a distribution $f$ each $f_n=\Fc^{-1}(\varphi_n\hat f)$
is a Schwartz function and $f=\sum_n f_n$. Then
\[
  \|f\|_{B^s_{p,q}}
    \eqdef\bigl\|(2^{ns}\|f_n\|_{L^p})_{n\geq0}\bigr\|_{\ell^q},
      \qquad\text{and}\qquad
  \|f\|_{F^s_{p,q}}
    \eqdef\bigl\|\|(2^{ns}f_n)_{n\geq0}\|_{\ell^q}\bigr\|_{L^p},
\]
($p<\infty$ for the $F^s_{p,q}$ norm).
Notice that to define the $B^s_{p,q}$ norm, the
Littlewood--Paley decomposition is first averaged
over position, and then over frequencies, while
the opposite happens for the $F^s_{p,q}$ norm.
We define the spaces $B^s_{p,q}(\R^d)$ and $F^s_{p,q}(\R^d)$,
with $s\in\R$ and $1\leq p,q\leq\infty$ (with the exception
of $p=\infty$ for the $F$--space) as the \emph{closure of the
Schwartz space} with respect to the above norms. In this
way the spaces we obtain are separable, regardless of the
index (but this is an issue only if $p=\infty$ or $q=\infty$).
Since $\|\cdot\|_{B^s_{p,p}}=\|\cdot\|_{F^s_{p,p}}$ for
all $p$, we define $F^s_{\infty,\infty}\eqdef B^s_{\infty,\infty}$.
The spaces obtained do not depend on the band--limited
decomposition (although the norms do), and different choices
of the decomposition give raise to equivalent norms.
\subsubsection{Definition via the difference operator}

Given $\alpha>0$, we shall denote by $C^\alpha(\R^d)$
the standard H\"older space, namely the space of
functions with $[\alpha]$ derivatives such that
the derivatives of order $[\alpha]$ are
H\"older continuous of exponent $(\alpha-[\alpha])$.

A special role in this paper is played by the
\emph{Zygmund spaces}
$\Cs^\alpha_b(\R^d)=B^\alpha_{\infty,\infty}(\R^d)$
that, for non--integer values of $\alpha$, coincide with
the (separable version of the) H\"older spaces. With
this in mind, we recall an alternative definition of Besov
spaces that is better suited for our purposes (see
\cite[Theorem 2.5.12]{Tri1983} or
\cite[Theorem 2.6.1]{Tri1992} for further details).
Define
\begin{equation}\label{e:discrete_increments}
  \begin{gathered}
    (\Delta_h^1f)(x)
      = f(x+h)-f(x),\\
    (\Delta_h^nf)(x)
      = \Delta_h^1(\Delta_h^{n-1}f)(x)
      = \sum_{j=0}^n (-1)^{n-j}\binom{n}{j} f(x+jh)
  \end{gathered}
\end{equation}
then the following norms, for $s>0$, $1\leq p\leq\infty$,
$1\leq q\leq\infty$,
\[
  \|f\|_{L^p} + [f]_{B^s_{p,q}}
\]
are equivalent norms of $B_{p,q}^s(\R^d)$ for the
given range of parameters. Here we have set
\begin{equation}\label{e:bseminorm}
  [f]_{B^s_{p,q}}
    \eqdef \Bigl\|h\mapsto\frac{\|\Delta_h^m f\|_{L^p}}
      {|h|^{s}}\Bigr\|_{L^q(B_1(0);\frac{dh}{|h|^d})}.    
\end{equation}
where $m$ is any integer such that $s<m$,
and $B_1(0)$ is the unit ball in $\R^d$.

A similar definition can be given for Triebel--Lizorkin
spaces (see \cite[Theorem 2.5.10]{Tri1983}).
Let $1\leq p<\infty$, $1\leq q\leq\infty$, 
and $s>\frac{d}{p\wedge q}$, and set
\begin{equation}\label{e:fseminorm}
  [f]_{F^s_{p,q}}
    \eqdef \Bigl\|x\mapsto
      \|h\mapsto |h|^{-s}\Delta_h^mf(x)\|_{L^q(\R^d;\frac{dh}{|h|^d})}
      \Bigr\|_{L^p}
\end{equation}
where $m$ is an integer $m>s$.
Then $\|f\|_{L^p} + [f]_{F^s_{p,q}}$ is an
equivalent norm in $F^s_{p,q}(\R^d)$.
Unfortunately the representation in terms of
differences only holds for $s$ large.
\subsubsection{Besov spaces on domains}

Given a bounded domain $D$ with smooth boundary, define
for $s\in\R$ and $1\leq p,q\leq\infty$,
\[
  B^s_{p,q}(D)
    \eqdef\{f: \text{there is }g\in B^s_{p,q}
      \text{ such that }g|_D=f\},
\]
with norm
\[
  \|f\|_{B^s_{p,q}(D)}
    \eqdef\inf\{\|g\|_{B^s_{p,q}}:g\in B^s_{p,q}(\R^d)
      \text{ and }g|_D=f\}.
\]
See \cite[Chapter 3]{Tri1983} for further details.
\subsection{Smoothing results}

This lemma, in a weaker form, is implicitly given in \cite{DebRom2014} (see also \cite[Lemma 4.1]{Rom2016b}, and see
\cite[Lemma 2.1]{DebFou2013} for a real analytic proof).
\begin{lemma}[fractional integration by parts]\label{l:smoothing1}
  Let $X$ be a $\R^d$--valued random variable. If there
  are an integer $m\geq1$,
  a real number $s>0$, a real $\alpha>0$,
  with $\alpha<s<m$, and a constant $K>0$
  such that for every $\phi\in \Cs^\alpha_b(\R^d)$ and
  $h\in\R^d$, with $|h|\leq 1$,
  \[
    \E[\Delta_h^m\phi(X)]
      \leq K|h|^s\|\phi\|_{\Cs_b^\alpha},
  \]
  then $X$ has a density $f_X$ with respect to the
  Lebesgue measure on $\R^d$.
  Moreover $f_X\in B^{s-\alpha}_{1,\infty}(\R^d)$ and
  \begin{equation}\label{e:smoothing}
    \|f\|_{B^{s-\alpha}_{1,\infty}}
      \lesssim (1+K).
  \end{equation}
\end{lemma}
\begin{proof}
  Let $\mu$ be the law of $X$.
	Fix a smooth function $\phi$.
	Let $(\varphi_\epsilon)_{\epsilon>0}$ be a smoothing kernel,
	namely $\varphi_\epsilon = \epsilon^{-d}\varphi(x/\epsilon)$,
	with $\varphi\in C^\infty_c(\R^d)$, $0\leq\varphi\leq1$, and
	$\int_{\R^d}\varphi(x)\,dx = 1$. Let $f_\epsilon=\varphi_\epsilon
	\star\mu$, then easy computations show that $f_\epsilon\geq0$,
	$\int_{\R^d}f_\epsilon(x)\,dx=1$ and that
	\[
	  \Bigl|\int_{\R^d}\Delta_h^m\phi(x) f_\epsilon(x)\,dx\Bigr|
	    = \Bigl|\int\varphi_\epsilon(x)
        \E[\Delta_h^m\phi(x-X)]\,dx\Bigr|
	    \leq K|h|^s\|\phi\|_{\Cs^\alpha_b}.
	\]
	On the other hand, by a discrete integration by parts,
	\begin{equation}\label{e:dibp}
	  \int_{\R^d}\Delta_h^m\phi(x) f_\epsilon(x)\,dx
	    = \int_{\R^d}\Delta_{-h}^mf_\epsilon(x) \phi(x)\,dx.
	\end{equation}
	Set $g_\epsilon=(I-\Delta_d)^{-\alpha/2}f_\epsilon$,
	and $\psi = (I-\Delta_d)^{\alpha/2}\phi$, where $\Delta_d$
	is the $d$--dimensional Laplace operator.
	We have by \cite[Theorem 10.1]{AroSmi1961} that
	$\|g_\epsilon\|_{L^1}\leq c\|f_\epsilon\|_{L^1}$. Moreover,
	by \cite[Theorem 2.5.7,Remark 2.2.2/3]{Tri1983}),
	we know that for $\alpha>0$,
	$\Cs^\alpha_b(\R^d) = B^\alpha_{\infty,\infty}(\R^d)$,
	and by \cite[Theorem 2.3.8]{Tri1983} we know that
	$(I - \Delta_d)^{-\alpha/2}$ is a continuous operator
	from $B_{\infty,\infty}^0(\R^d)$ to
	$B_{\infty,\infty}^\alpha(\R^d)$.
	Hence, by \eqref{e:dibp} it follows that
	\[
	  \int_{\R^d}\Delta_h^mg_\epsilon(x)\psi(x)\,dx
	    = \int_{\R^d}\Delta_h^mf_\epsilon(x)\phi(x)\,dx
	    \leq K|h|^s\|\phi\|_{\Cs^\alpha_b}
	    \leq c K|h|^s\|\psi\|_{B^0_{\infty,\infty}}
	\]
	Notice that by
  \cite[Theorem 2.11.2 and Remark 2.11.2/2]{Tri1983},
	$B_{\infty,\infty}^0(\R^d)$ is the dual
	of $B_{1,1}^0(\R^d)$, moreover
	$B_{1,1}^0(\R^d)\hookrightarrow L^1(\R^d)$
	by \cite[Proposition 2.5.7]{Tri1983},
	therefore
  $\|\Delta_h^mg_\epsilon\|_{L^1}\leq\|\Delta_h^mg_\epsilon\|_{B^0_{1,1}}
	\leq c K|h|^s$,
	hence $\|g_\epsilon\|_{B^s_{1,\infty}}\leq c(1+K)$.
	Again, since $(I - \Delta_d)^{\alpha/2}$ maps continuously
	$B_{1,\infty}^s(\R^d)$ into
	$B_{1,\infty}^{s-\alpha}(\R^d)$, it finally follows
	that $\|f_\epsilon\|_{B^{s-\alpha}_{1,\infty}}\leq c\|g_\epsilon\|_{B^s_{1,\infty}}$.
    
	By Sobolev's embeddings and
	\cite[formula 2.2.2/(18)]{Tri1983},
	we have for every $r<s-\alpha$ and $1\leq p\leq d/(d-r)$ that
	$B_{\smash{1,\infty}}^{s-\alpha}(\R^d)\hookrightarrow
	B_{1,1}^r(\R^d) = W^{r,1}(\R^d)\hookrightarrow L^p(\R^d)$.
	In particular, $(f_\epsilon)_{\epsilon>0}$ is uniformly integrable
	in $L^1(\R^d)$, therefore there is $f_X$ such that $\mu = f_X\,dx$
	and $(f_\epsilon)_{\epsilon>0}$ converges weakly in $L^1(\R^d)$
	to $f_\mu$. Formula \eqref{e:smoothing} follows by semi--continuity.
\end{proof}
By the proof, it is clear that the Lemma applies in the
case of a positive finite measure (so, not necessarily
of mass one).

The properties of Besov spaces we have used hold regardless of
the summability parameters of the spaces, thus one can show
with arguments entirely similar with those above the
following result.
\begin{corollary}\label{c:smoothing1}
  Let $X$ be a $\R^d$--valued random variable. If there
  are an integer $m\geq1$, numbers $\alpha$, $s$, $p$,
  $q$, with $0<\alpha<s<m$, $1<p\leq\infty$ and
  $1\leq q\leq\infty$, and a constant $K>0$
  such that for every $\phi\in B^\alpha_{p,q}(\R^d)$
  and $h\in\R^d$, with $|h|\leq 1$,
  \[
    \E[\Delta_h^m\phi(X)]
      \leq K|h|^s\|\phi\|_{B^\alpha_{p,q}},
  \]
  then $X$ has a density $f_X$ with respect to the
  Lebesgue measure on $\R^d$. Moreover, for every
  $r<s-\alpha$, $f_X\in B^r_{p',\infty}(\R^d)$ and
  \[
    \|f\|_{B^r_{p',\infty}}
      \lesssim (1+K),
  \]
  where $p'$ is the conjugate H\"older exponent of $p$.  
\end{corollary}
We will also use the following version of the smoothing lemma
in the paper.
\begin{lemma}\label{l:smoothing2}
  Let $1<p<\infty$ and $f\in L^p(\R^d)$. Assume there are
  an integer $m\geq1$, two real numbers $s>0$ and $\alpha>0$,
  with $\alpha<s<m$, and a constant $K>0$
  such that for every $\phi\in C^\infty_c(\R^d)$ and
  $h\in\R^d$, with $|h|\leq 1$,
  \[
    \Bigl|\int_{\R^d} \Delta_h^m\phi(x)f(x)\,dx\Bigr|
      \leq K|h|^s\|\phi\|_{F^\alpha_{q,\infty}},
  \]
  where $q$ is the H\"older conjugate exponent of $p$.
  Then $f\in B^{s-\alpha}_{p,\infty}$ and
  \begin{equation}\label{e:smoothing2}
    \|f\|_{B^{s-\alpha}_{p,\infty}}
      \lesssim (\|f\|_{L^p} + K).
  \end{equation}
\end{lemma}
\begin{proof}
  Fix a smooth function $\phi$ and set $g=(I-\Delta_d)^{-\alpha/2}f$,
  $\psi=(I-\Delta_d)^{\alpha/2}\phi$, then by integration by parts,
  \[
    \begin{multlined}[.85\linewidth]
      \Bigl|\int_{\R^d}(\Delta_h^mg)(x)\psi(x)\,dx\Bigr|
        = \Bigl|\int_{\R^d}g(x)(\Delta_{-h}^m\psi)(x)\,dx\Bigr| =\\
        = \Bigl|\int_{\R^d}f(x)(\Delta_{-h}^m\phi)(x)\,dx\Bigr|
        \leq K|h|^s\|\phi\|_{F^\alpha_{q,\infty}}
        \leq K|h|^s\|\psi\|_{F^0_{q,\infty}},
    \end{multlined}
  \]
  since by \cite[Theorem 2.3.8]{Tri1983},
  $(I-\Delta_d)^{-\alpha/2}$ maps $F^\alpha_{q,\infty}$
  into $F^0_{q,\infty}$.
  
  We know by \cite[Theorem 2.11.2]{Tri1983} that the space
  $F^0_{q,\infty}$ is the dual of $F^0_{p,1}$, therefore
  we deduce from the inequality above that
  $\|\Delta_h^m g\|_{F^0_{p,1}}\leq K|h|^s$. On the other hand
  we know by \cite[Proposition 2.5.7]{Tri1983} that $F^0_{p,1}\subset L^p$,
  hence $\|\Delta_h^mg\|_{L^p}\lesssim K|h|^s$.
  
  Finally, by \cite{AroSmi1961}, $\|g\|_{L^p}\lesssim
  \|f\|_{L^p}$, and in conclusion $g\in B^s_{p,\infty}$ and
  $\|g\|_{B^s_{p,\infty}}\lesssim(K+\|f\|_{L^p})$,
  and then \eqref{e:smoothing2} follows, since $\|f\|_{B^{s-\alpha}_{p,\infty}}
  =\|(I - \Delta_d)^{\alpha/2}g\|_{B^{s-\alpha}_{p,\infty}}
  =\|g\|_{B^s_{p,\infty}}$.
\end{proof}
\begin{remark}
  Since $F^\alpha_{\infty,\infty}=B^\alpha_{\infty,\infty}$
  for all $\alpha\in\R$, the case $p=\infty$
  in the lemma above is already covered by Lemma~\ref{l:smoothing1}.
\end{remark}
\bibliographystyle{amsplain}

\begin{thebibliography}{10}

\bibitem{Alt2017}
Randolf Altmayer, \emph{Estimating occupation time functionals},
  \arxiv{1706.03418}, 2017.

\bibitem{AroSmi1961}
Nachman Aronszajn and Kennan~T. Smith, \emph{Theory of {B}essel potentials
  {I}}, Ann. Inst. Fourier (Grenoble) \textbf{11} (1961), 385--475.
  \MR{0143935}

\bibitem{BanKru2016}
David Ba\~nos and Paul Kr\"uhner, \emph{Optimal density bounds for marginals of
  {I}t\^o processes}, Commun. Stoch. Anal. \textbf{10} (2016), no.~2, 131--150.
  \MR{3605421}

\bibitem{BanKru2017}
\bysame, \emph{H\"older continuous densities of solutions of {SDE}s with
  measurable and path dependent drift coefficients}, Stochastic Process. Appl.
  \textbf{127} (2017), no.~6, 1785--1799. \MR{3646431}

\bibitem{BalCar2012}
Vlad Bally and Lucia Caramellino, \emph{Regularity of probability laws by using
  an interpolation method}, 2012, \arxiv{1211.0052}.

\bibitem{BalCar2014}
\bysame, \emph{Convergence and regularity of probability laws by using an
  interpolation method}, \arxiv{1409.3118}, 2014.

\bibitem{BalCar2016}
\bysame, \emph{Integration by parts formulas, {M}alliavin calculus, and
  regularity of probability laws}, Stochastic integration by parts and
  functional {I}t\^o calculus, Adv. Courses Math. CRM Barcelona,
  Birkh\"auser/Springer, [Cham], 2016, pp.~1--114. \MR{3497714}

\bibitem{CanFriGas2017}
Giuseppe Cannizzaro, Peter~K. Friz, and Paul Gassiat, \emph{Malliavin calculus
  for regularity structures: the case of g{PAM}}, J. Funct. Anal. \textbf{272}
  (2017), no.~1, 363--419. \MR{3567508}

\bibitem{CatCho2013}
R{\'e}mi Catellier and Khalil Chouk, \emph{Paracontrolled distributions and the
  3-dimensional stochastic quantization equation}, \arxiv{1310.6869}, 2013.

\bibitem{CheWan2016}
Zhen-Qing Chen and Longmin Wang, \emph{Uniqueness of stable processes with
  drift}, Proc. Amer. Math. Soc. \textbf{144} (2016), no.~6, 2661--2675.
  \MR{3477084}

\bibitem{DapDeb2003}
Giuseppe Da~Prato and Arnaud Debussche, \emph{Ergodicity for the 3{D}
  stochastic {N}avier-{S}tokes equations}, J. Math. Pures Appl. (9) \textbf{82}
  (2003), no.~8, 877--947. \MR{2005200}

\bibitem{Dem2011}
Stefano De~Marco, \emph{Smoothness and asymptotic estimates of densities for
  {SDE}s with locally smooth coefficients and applications to square root-type
  diffusions}, Ann. Appl. Probab. \textbf{21} (2011), no.~4, 1282--1321.
  \MR{2857449}

\bibitem{DebFou2013}
Arnaud Debussche and Nicolas Fournier, \emph{Existence of densities for
  stable-like driven {SDE}'s with {H}\"older continuous coefficients}, J.
  Funct. Anal. \textbf{264} (2013), no.~8, 1757--1778. \MR{3022725}

\bibitem{DebRom2014}
Arnaud Debussche and Marco Romito, \emph{Existence of densities for the 3{D}
  {N}avier--{S}tokes equations driven by {G}aussian noise}, Probab. Theory
  Related Fields \textbf{158} (2014), no.~3-4, 575--596. \MR{3176359}

\bibitem{DubRev2015p}
Romain Duboscq and Anthony R{\'e}veillac, \emph{Stochastic regularization
  effects of semi-martingales on random functions}, 2015, \arxiv{1507.05579}.

\bibitem{FedFla2013}
Ennio Fedrizzi and Franco Flandoli, \emph{H{\"o}lder flow and differentiability
  for {SDE}s with nonregular drift}, Stoch. Anal. Appl. \textbf{31} (2013),
  no.~4, 708--736. \MR{3175794}

\bibitem{FlaGubPri2010}
Franco Flandoli, Massimiliano Gubinelli, and Enrico Priola,
  \emph{Well-posedness of the transport equation by stochastic perturbation},
  Invent. Math. \textbf{180} (2010), no.~1, 1--53. \MR{2593276}

\bibitem{Fou2015}
Nicolas Fournier, \emph{Finiteness of entropy for the homogeneous {B}oltzmann
  equation with measure initial condition}, Ann. Appl. Probab. \textbf{25}
  (2015), no.~2, 860--897. \MR{3313757}

\bibitem{FouPri2010}
Nicolas Fournier and Jacques Printems, \emph{Absolute continuity for some
  one--dimensional processes}, Bernoulli \textbf{16} (2010), no.~2, 343--360.
  \MR{2668905}

\bibitem{Hai2014}
Martin Hairer, \emph{A theory of regularity structures}, Invent. Math.
  \textbf{198} (2014), no.~2, 269--504. \MR{3274562}

\bibitem{Hai2015}
\bysame, \emph{Introduction to regularity structures}, Braz. J. Probab. Stat.
  \textbf{29} (2015), no.~2, 175--210. \MR{3336866}

\bibitem{Hai2016}
\bysame, \emph{Regularity structures and the dynamical {$\Phi^4_3$} model},
  Current developments in mathematics 2014, Int. Press, Somerville, MA, 2016,
  pp.~1--49. \MR{3468250}

\bibitem{HayKohYuk2013}
Masafumi Hayashi, Arturo Kohatsu-Higa, and G{\^o}~Y{\^u}ki, \emph{Local
  {H}\"older continuity property of the densities of solutions of {SDE}s with
  singular coefficients}, J. Theoret. Probab. \textbf{26} (2013), no.~4,
  1117--1134. \MR{3119987}

\bibitem{HayKohYuk2014}
\bysame, \emph{H\"older continuity property of the densities of {SDE}s with
  singular drift coefficients}, Electron. J. Probab. \textbf{19} (2014), no.
  77, 22. \MR{3256877}

\bibitem{Hua2015}
Lorick Huang, \emph{Density estimates for {SDE}s driven by tempered stable
  processes}, \arxiv{1504.04183}, 2015.

\bibitem{KloPla1992}
Peter~E. Kloeden and Eckhard Platen, \emph{Numerical solution of stochastic
  differential equations}, Applications of Mathematics (New York), vol.~23,
  Springer-Verlag, Berlin, 1992. \MR{1214374}

\bibitem{KnoKul2017}
Victoria Knopova and Alexei Kulik, \emph{Parametrix construction of the
  transition probability density of the solution to an {SDE} driven by
  $\alpha$-stable noise}, \arxiv{1412.8732}, to appear on the Annales de
  l'Institut Henri Poincar\'e, 2014.

\bibitem{KohLi2016}
Arturo Kohatsu-Higa and Libo Li, \emph{Regularity of the density of a
  stable-like driven {SDE} with {H}\"older continuous coefficients}, Stoch.
  Anal. Appl. \textbf{34} (2016), no.~6, 979--1024. \MR{3544166}

\bibitem{KryRoc2005}
Nicolai~V. Krylov and Michael R{\"o}ckner, \emph{Strong solutions of stochastic
  equations with singular time dependent drift}, Probab. Theory Related Fields
  \textbf{131} (2005), no.~2, 154--196. \MR{2117951}

\bibitem{Kus2010}
Seiichiro Kusuoka, \emph{Existence of densities of solutions of stochastic
  differential equations by {M}alliavin calculus}, J. Funct. Anal. \textbf{258}
  (2010), no.~3, 758--784. \MR{2558176}

\bibitem{LebLio2008}
Claude Le~Bris and Pierre-Louis Lions, \emph{Existence and uniqueness of
  solutions to {F}okker-{P}lanck type equations with irregular coefficients},
  Comm. Partial Differential Equations \textbf{33} (2008), no.~7-9, 1272--1317.
  \MR{2450159}

\bibitem{Mal1978}
Paul Malliavin, \emph{Stochastic calculus of variation and hypoelliptic
  operators}, Proceedings of the {I}nternational {S}ymposium on {S}tochastic
  {D}ifferential {E}quations ({R}es. {I}nst. {M}ath. {S}ci., {K}yoto {U}niv.,
  {K}yoto, 1976) (New York-Chichester-Brisbane), Wiley, 1978, pp.~195--263.
  \MR{536013}

\bibitem{MeyPro2010}
Thilo Meyer-Brandis and Frank Proske, \emph{Construction of strong solutions of
  {SDE}'s via {M}alliavin calculus}, J. Funct. Anal. \textbf{258} (2010),
  no.~11, 3922--3953. \MR{2606880}

\bibitem{MouWeb2015}
Jean-Christophe Mourrat and Hendrik Weber, \emph{Global well-posedness of the
  dynamic {$\Phi^4$} model in the plane}, 2015, to appear on Annals of
  Probability.

\bibitem{MouWeb2017}
\bysame, \emph{Global well-posedness of the dynamic {$\Phi^4$} model in the
  plane}, Ann. Probab. \textbf{45} (2017), no.~4, 2398--2476. \MR{3693966}

\bibitem{MouWebXu2016}
Jean-Christophe Mourrat, Hendrik Weber, and Weijun Xu, \emph{Construction of
  $\phi^4_3$ diagrams for pedestrians}, \arxiv{1610.08897}, 2016.

\bibitem{MouWeb2016}
Jean-Cristophe Mourrat and Hendrik Weber, \emph{The dynamic $\phi^4_3$ model
  comes down from infinity}, \arxiv{1601.01234}, 2016.

\bibitem{Nua2006}
David Nualart, \emph{The {M}alliavin calculus and related topics}, second ed.,
  Probability and its Applications (New York), Springer-Verlag, Berlin, Berlin,
  2006. \MR{2200233}

\bibitem{Rom2014a}
Marco Romito, \emph{Unconditional existence of densities for the
  {N}avier-{S}tokes equations with noise}, Mathematical analysis of viscous
  incompressible fluid, RIMS K{\^o}ky{\^u}roku, vol. 1905, Kyoto University,
  2014, pp.~5--17.

\bibitem{Rom2014}
\bysame, \emph{Uniqueness and blow-up for a stochastic viscous dyadic model},
  Probab. Theory Related Fields \textbf{158} (2014), no.~3-4, 895--924.
  \MR{3176368}

\bibitem{Rom2016}
\bysame, \emph{H{\"o}lder continuity of the densities for the
  {N}avier--{S}tokes equations with noise}, Stoch. Partial Differ. Equ. Anal.
  Comput. \textbf{4} (2016), no.~3, 691--711. \MR{3538013}

\bibitem{Rom2016a}
\bysame, \emph{Some probabilistic topics in the {N}avier--{S}tokes equations},
  Recent {P}rogress in the {T}heory of the {E}uler and {N}avier--{S}tokes
  {E}quations (James~C. Robinson, Jos{\'e}~L. Rodrigo, Witold Sadowski, and
  Alejandro Vidal-L{\'o}pez, eds.), London Math. Soc. Lecture Note Ser., vol.
  430, Cambridge Univ. Press, Cambridge, 2016, pp.~175--232. \MR{3497693}

\bibitem{Rom2016b}
\bysame, \emph{Time regularity of the densities for the {N}avier--{S}tokes
  equations with noise}, J. Evol. Equations \textbf{16} (2016), no.~3,
  503--518. \MR{3551234}

\bibitem{SanSus2015}
Marta Sanz-Sol{\'e} and Andr{\'e} S{\"u}{\ss}, \emph{Absolute continuity for
  {SPDE}s with irregular fundamental solution}, Electron. Commun. Probab.
  \textbf{20} (2015), no. 14, 11. \MR{3314649}

\bibitem{SanSus2015p}
\bysame, \emph{Non elliptic {SPDE}s and ambit fields: existence of densities},
  2015, \arxiv{1502.02386}.

\bibitem{SchSztWan2012}
Ren{\'e}~L. Schilling, Pawe{\l} Sztonyk, and Jian Wang, \emph{Coupling property
  and gradient estimates of {L}\'evy processes via the symbol}, Bernoulli
  \textbf{18} (2012), no.~4, 1128--1149. \MR{2995789}

\bibitem{Tri1983}
Hans Triebel, \emph{Theory of function spaces}, Monographs in Mathematics,
  vol.~78, Birkh\"auser Verlag, Basel, 1983. \MR{781540}

\bibitem{Tri1992}
\bysame, \emph{Theory of function spaces. {II}}, Monographs in Mathematics,
  vol.~84, Birkh\"auser Verlag, Basel, 1992. \MR{1163193}
\end{thebibliography}

\end{document}